\documentclass[11pt]{article}
\usepackage{amsmath,amssymb,amsthm}
\usepackage{float}
\floatstyle{ruled}
\newfloat{algorithm}{tbp}{loa}
\floatname{algorithm}{Algorithm}
\usepackage{enumitem}
\usepackage[expansion=false]{microtype}
\usepackage{booktabs,array}
\newcolumntype{P}[1]{>{\raggedright\arraybackslash}p{#1}}
\usepackage[numbers,sort&compress]{natbib}
\usepackage[hidelinks]{hyperref}
\usepackage[margin=1in]{geometry}
\usepackage{standalone}
\usepackage{subcaption}
\usepackage{tikz}

\newtheorem{theorem}{Theorem}[section]
\newtheorem{lemma}[theorem]{Lemma}
\newtheorem{corollary}[theorem]{Corollary}
\newtheorem{proposition}[theorem]{Proposition}
\newtheorem{definition}[theorem]{Definition}
\newtheorem{remark}[theorem]{Remark}
\newcommand{\R}{\mathbb{R}}
\newcommand{\Z}{\mathbb{Z}}
\newcommand{\rowsp}{\mathrm{rowspace}}

\newcommand{\calN}{\mathcal{N}}
\newcommand{\calG}{\mathcal{G}}
\newcommand{\calF}{\mathcal{F}}
\newcommand{\Span}{\mathrm{span}}
\newcommand{\rank}{\mathrm{rank}}
\newcommand{\calH}{\mathcal{H}}
\newcommand{\ZR}{\Z^{p}\times\R^{q}}

\title{Curvature batching for integer and\\mixed-integer quadratic programming}
\author{Cinar Ari\thanks{Grado Department of Industrial and Systems Engineering,
Virginia Tech.} \and
Robert Hildebrand\thanks{Grado Department of Industrial and Systems Engineering,
Virginia Tech. Email: \texttt{rhil@vt.edu}}}
\date{}

\begin{document}
\maketitle

\begin{abstract}
We give an explicit algorithm for integer quadratic programming
with running time $(nL)^{O(n^2)}\mathrm{poly}(\varphi)$, where $n$
is the number of variables, $L$ bounds the absolute entries of the
constraint and quadratic matrices, and $\varphi$ is the total input
encoding length. The linear objective and constraint right-hand
sides enter only through $\varphi$.

The algorithm refines Lokshtanov's branching framework. Small integer
kernel directions identify optima near constraint boundaries. When
all chosen basis directions have nonnegative curvature, their gradient
values can be imposed in one batch. The quadratic form then vanishes
on the remaining kernel, and the residual problem is an integer linear
program. Thus determinant growth after the batch does not feed into
further branching. Over a polytope, the method extends to mixed-integer
quadratic programming with $q$ continuous variables in time
$(nL)^{O(n^2(q+1)^2)}\mathrm{poly}(\varphi)$, without a convexity
assumption on the continuous block.

We also decide and certify integer and mixed-integer unboundedness in
$(m+n)^{O(n)}\mathrm{poly}(\varphi)$ time for arbitrary rational
polyhedra with $m$ inequalities. Noncopositivity on polyhedral cones
and integer unboundedness are W[1]-hard in the dimension and admit no
$(m+\varphi)^{o(n)}$ algorithm under ETH, matching the linear order of
dimension in the recession-test exponent. Companion results give a
right-hand-side-independent bound for concave integer minimization
and explicit runtimes for applications of IQP.
\end{abstract}
\section{Introduction}
We study integer quadratic programming (IQP),
\begin{equation}\label{eq:iqp}
 \min\{f(x)=x^TQx+c^Tx:Ax\le b,\ x\in\Z^n\},
\end{equation}
where $Q\in\Z^{n\times n}$ is symmetric, $A\in\Z^{m\times n}$,
$b\in\Z^m$, and $c\in\Z^n$. No definiteness assumption is imposed
on $Q$. Let $L_A=\max(1,\|A\|_{\max})$,
$L_Q=\max(1,\|Q\|_{\max})$, and $L=\max(L_A,L_Q)$, where
$\|\cdot\|_{\max}$ is the largest absolute entry. Write $\varphi$
for the total input encoding length and $\Delta(A)$ for the maximum
absolute square subdeterminant, including the empty determinant $1$.
The dependence on the magnitudes of $b$ and $c$ is therefore
polynomial in their encoding lengths, rather than in their values.

Lokshtanov~\cite{Lokshtanov17} proved that IQP is fixed-parameter
tractable in $(n,L)$, with an unspecified parametric factor; Zemmer
\cite{Zemmer17} independently developed a related branching approach.
We make the running time explicit and replace sequential gradient
branching by one batch on every root-to-leaf path.

\subsection{Results}
\begin{theorem}[Curvature batch bound]\label{thm:main}
IQP~\eqref{eq:iqp} can be solved in time
\[
 (mn\Delta(A)L_A L_Q)^{O(n)}\mathrm{poly}(\varphi)
\]
when $m\ge1$. Removing duplicate coefficient rows gives the bound
\[
 \bigl(n(2L_A+1)^n\Delta(A)L_Q\bigr)^{O(n)}
 \mathrm{poly}(\varphi)
 = (nL)^{O(n^2)}\mathrm{poly}(\varphi).
\]
For systems with no inequalities, replace $m$ by $\max(1,m)$
in the first display.
\end{theorem}
The algorithm reports infeasibility or unboundedness when appropriate.
For a finite optimum, it returns an optimal integer point. The refined
bound separates the original constraint geometry from the objective
coefficients. In particular, it yields the following consequence.

\begin{corollary}[Totally unimodular constraints]\label{cor:tu}
If $A$ is totally unimodular, IQP can be solved in
$(\max(1,m)nL_Q)^{O(n)}\mathrm{poly}(\varphi)$ time. After removing
duplicate rows this is
$2^{O(n^2)}(nL_Q)^{O(n)}\mathrm{poly}(\varphi)$.
\end{corollary}

We also consider mixed-integer quadratic programming (MIQP),
\begin{equation}\label{eq:miqp0}
 \min\{x^TQx+c^Tx:Ax\le b,\ x\in\Z^p\times\R^q\},
 \qquad n=p+q.
\end{equation}
Over a polytope, batching solves this problem in
$(nL)^{O(n^2(q+1)^2)}\mathrm{poly}(\varphi)$ time, without a
convexity assumption on any block of $Q$
(Theorem~\ref{thm:m-main}). At $q=0$ it reduces to the pure algorithm.

Unboundedness has a different parameter dependence. For arbitrary
rational polyhedra, we classify pure and mixed instances as infeasible,
unbounded below, or having a finite attained optimum in
$(m+n)^{O(n)}\mathrm{poly}(\varphi)$ time
(Theorem~\ref{thm:unb-main} and Corollary~\ref{cor:m-unb}).
The test has no coefficient-magnitude parameter, so finite-value
optimization remains the unresolved part of polynomial-time IQP in
every fixed dimension. A finite value does not imply a bounded
feasible polyhedron; the MIQP batching result specifically assumes
a polytope.

The recession exponent has matching linear order under ETH:
noncopositivity over polyhedral cones and IQP unboundedness are
W[1]-hard in $n$ and admit no $(m+\varphi)^{o(n)}$ algorithm
(Theorems~\ref{thm:lb-copos} and~\ref{thm:lb-unbounded}). The same
lower bound holds for the threshold problem on bounded IQP
instances (Proposition~\ref{prop:lb-integer}). The continuous
hardness ancestor is Knauer--K\"onig--Werner's norm-maximization
result~\cite{KKW15}. Herrmann~\cite{Herrmann26} proves dimension-parameterized
IQP hardness using a related parabola gadget. Our construction supplies
an explicit continuous gap and integer witnesses for the cone and
unboundedness transfers; bounded IQP hardness alone is not the novelty.

\begin{table}[t]
\centering\small
\renewcommand{\arraystretch}{1.25}
\begin{tabular}{P{0.24\textwidth}P{0.40\textwidth}P{0.26\textwidth}}
\toprule
Problem and domain & Upper bound & Dimension lower bound \\
\midrule
IQP, arbitrary polyhedron &
$(nL)^{O(n^2)}\mathrm{poly}(\varphi)$ &
W[1]-hard; no $(m+\varphi)^{o(n)}$ under ETH, even on polytopes \\
MIQP, polytope &
$(nL)^{O(n^2(q+1)^2)}\mathrm{poly}(\varphi)$ &
Inherits the IQP bound at $q=0$ \\
(M)IQP unboundedness &
$(m+n)^{O(n)}\mathrm{poly}(\varphi)$ &
W[1]-hard; no $(m+\varphi)^{o(n)}$ under ETH \\
Noncopositivity on cones &
$(m+n)^{O(n)}\mathrm{poly}(\varphi)$ &
W[1]-hard; no $(m+\varphi)^{o(n)}$ under ETH \\
Concave integer minimization, polytope &
$((m+n)n\varphi_A)^{O(n)}\mathrm{poly}(\varphi)$ &
Includes bounded concave IQP hardness \\
\bottomrule
\end{tabular}
\caption{Main bounds. Optimization and unboundedness results are proved in
Sections~\ref{sec:batch}--\ref{sec:lb}; the concave bound is
Proposition~\ref{prop:concave}. Here $\varphi_A$ is the maximum row
encoding length, with a lower bound of $1$.}
\label{tab:results}
\end{table}

\subsection{Why batching works}\label{subsec:Approach}
A recursion node carries independent equations $Cx=d$. We compute
integer vectors $y_1,\ldots,y_r$ spanning $V=\ker C$ over $\R$,
with $\|y_i\|_\infty\le\Delta(C)$. These directions preserve the
node equations. If an optimal point $x^*$ cannot move both one step
forward and one step backward along a direction, an original
inequality has bounded slack at $x^*$. Branching on that slack fixes
a nearby level $a_j^Tx=b'_j$ and reduces the dimension.

If both steps are feasible, the identity
\[
 f(x^*\pm y)-f(x^*)
 =\pm(2y^TQx^*+c^Ty)+y^TQy
\]
forces nonnegative curvature and a bounded gradient value. Thus a
negative-curvature basis vector forces a constraint branch. Otherwise,
the branch containing a deep optimum can impose all independent
gradient equations at once. These equations span the image of the
old kernel under $Q$, modulo the existing equations. Consequently
$Q$ vanishes on the new kernel, and the remaining objective is affine.

Each path therefore has at most $n$ constraint steps, one gradient
batch, and an ILP leaf. Before the batch, $C$ contains only original
rows, so $\Delta(C)\le\Delta(A)$. After the batch there is no further
quadratic branching. This removes the repeated determinant squaring
that drives the sequential bound in Appendix~\ref{sec:seq}.
Nonnegative curvature on the chosen basis does not imply positive
semidefiniteness on the whole kernel; the spanning argument is what
eliminates all remaining curvature.

\subsection{Related work and applications}
Integer linear programming is fixed-parameter tractable in the number
of variables, beginning with Lenstra~\cite{Lenstra83} and improved by
Kannan~\cite{Kannan87} and subsequent work
\cite{Dadush12,ReisRothvoss23}. Lenstra-type methods also apply to
convex and quasiconvex polynomial optimization
\cite{Heinz05,KhachiyanPorkolab00,HildebrandKoppe13}. For concave
minimization over a polytope, a minimum is attained at an integer-hull
vertex, permitting enumeration in fixed dimension
\cite{Hartmann89,CookHartmannKannanMcDiarmid92}.

The indefinite case combines these geometries. Del Pia, Dey, and
Molinaro~\cite{DelPiaDeyMolinaro17} proved NP membership for MIQP
in variable dimension, while Del Pia and Weismantel
\cite{DelPiaWeismantel14} established polynomial-time IQP in two
variables. Lokshtanov's $(n,L)$ parameterization allows coefficient
magnitudes to enter the parametric factor; our improvement concerns
that factor and does not resolve polynomial-time finite-value IQP
for every fixed dimension.

Herrmann~\cite{Herrmann26} reduces Independent Set to IQP with a
separable concave objective using $2k$ variables and secants of a
parabola. Adding upper bounds that retain the intended integer
witnesses also gives bounded instances. His dimension-linear,
polynomial-size reduction likewise transfers the usual ETH lower
bound. Section~\ref{sec:lb} develops an arbitrary-set chord gadget
from multicolored $k$-Sum, proves a continuous gap of $1/2$, and
tracks inertia through the cone and unboundedness transfers.

IQP formulations occur in network flow, production planning, and
lattice decoding, and as subroutines in parameterized graph algorithms.
Section~\ref{sec:applications} records the resulting bounds and the
formulation parameters on which they depend. Section~\ref{sec:concave}
gives a separate right-hand-side-independent construction for concave
integer minimization, using proximity and dyadic slack cells.

Section~\ref{sec:prelim} develops the shared branching framework;
Sections~\ref{sec:batch} and~\ref{sec:mixed} prove the pure and mixed
batching bounds. Section~\ref{sec:unbounded} treats unboundedness,
and Section~\ref{sec:lb} gives the lower bounds. The concave result
and applications follow, before the open problems in
Section~\ref{sec:discussion}. Appendix~\ref{sec:seq} contains the
sequential comparison.

\section{Branching framework and small kernel directions}
\label{sec:prelim}\label{sec:gradres}\label{sec:further}
We use the notation of the introduction. A node has the form
\begin{equation}\label{eq:iqp2}
 \min\{x^TQx+c^Tx:Ax\le b,\ Cx=d,\ x\in\Z^n\},
\end{equation}
where the rows of $C\in\Z^{k\times n}$ are independent and
$V=\ker C$ has dimension $r=n-k$. Initially $C$ is empty; supplied
equalities are included as pairs of opposite rows of $Ax\le b$
and hence counted in $L_A$ and $\Delta(A)$. The recursive equations
are additional restrictions. We first describe the finite-optimum
case; the test in Section~\ref{sec:unbounded} supplies the initial
classification needed by the full algorithm.

All matrix coefficient bounds use the maximum absolute entry.
The encoding length $\varphi_A\ge1$ is the maximum binary length of
one coefficient row of $A$, excluding its right-hand side;
$\varphi_{A,C}$ denotes the corresponding quantity for the combined
rows of $A$ and $C$. Total input length is always $\varphi$.
A right-hand-side-independent bound means that $b,d$ affect only the
fixed-degree polynomial in total encoding length. We take
$\Delta(C)\ge1$ by including the empty determinant, and interpret
an optimum over an empty feasible set as $+\infty$.

For a subspace $V$, the restricted bilinear form is
$Q|_V(u,v)=u^TQv$. If $B$ has columns forming a basis of $V$, its
matrix is $B^TQB$. Its inertia
$\nu(Q|_V)=(\nu_+,\nu_-,\nu_0)$ counts positive, negative, and
zero eigenvalues and is basis-independent by Sylvester's law
\cite{GolubVanLoan13}. In particular, $Q|_V=0$ means
$u^TQv=0$ for every $u,v\in V$.

\subsection{Small integer directions}\label{sec:adjugate}
We need integer vectors forming a basis of $\ker C$ over $\R$;
we do not require them to generate $\ker C\cap\Z^n$ as a lattice.
The following adjugate construction bounds every coordinate by a
minor of the current equality matrix.

\begin{lemma}[Adjugate basis construction]\label{lem:adjugate}
Let $C\in\Z^{k\times n}$ with $\operatorname{rank}(C)=k<n$.  There exist
linearly independent integer vectors $y_1,\ldots,y_{n-k}\in\ker(C)$ with
$\|y_i\|_\infty\le\Delta(C)$.
\end{lemma}

\begin{proof}
If $k=0$, take the standard coordinate vectors. Otherwise, given $C \in \mathbb{Z}^{k \times n}$ of rank $k < n$, pick any $k$ columns that form a
nonsingular submatrix and call it $C_S$, where $S \subseteq [n]$ is the corresponding index
set. Let $T = [n] \setminus S$ be the index set of the remaining $r = n-k$ columns. Define
$\delta := \det(C_S)$, and for each $t \in T$ let $C_t \in \mathbb{Z}^k$ denote the column of
$C$ at index $t$.

We define our proposed basis vectors for the kernel as follows: For each $t \in T$, define $y^{(t)} \in \mathbb{Z}^n$ by specifying its $S$- and
$T$-coordinate blocks separately:
\[
  y^{(t)}_S := -\,\mathrm{adj}(C_S)\, C_t \in \mathbb{Z}^k, \qquad
  y^{(t)}_T := \delta\, e_t \in \mathbb{Z}^r,
\]
where $e_t$ is the standard basis vector in $\mathbb{R}^r$ corresponding to index $t$.

Each $y^{(t)}$ lies in $\ker(C)$ because the adjugate identity
$C_S \cdot \mathrm{adj}(C_S) = \delta I$ gives
\[
  C\, y^{(t)} = C_S\, y^{(t)}_S + C_T\, y^{(t)}_T
             = -\delta\, C_t + \delta\, C_t = 0.
\]

The $r$ vectors are linearly independent since their $T$-blocks form $\delta \cdot I_r$, where $I_r$ is the identity matrix of dimension $r$. Linear independence and membership imply they are a basis for $\ker(C)$.

To show the norm bound holds, observe every entry of $y^{(t)}_T$ is $0$
or $\delta$, and $|\delta| \leq \Delta(C)$ by definition. Writing $M_{ij}$ for the $(i,j)$ minor of $C_S$, the $j$-th entry of
$\mathrm{adj}(C_S)\,C_t$ is
\[
  \sum_{i=1}^k (-1)^{i+j} M_{ij}\, [C_t]_i.
\]
This is precisely the cofactor expansion along column $j$ of the matrix formed by taking
$C_S$ and replacing its $j$-th column with $C_t$. Since this matrix has $k-1$ columns from $C_S$ and
one column $C_t$, it is a $k \times k$ submatrix of $C$, and its determinant is bounded
in absolute value by $\Delta(C)$. Thus $\|y^{(t)}\|_\infty \leq \Delta(C)$.
\end{proof}

\subsection{Shallow solutions and gradient values}
A feasible point is deep with respect to $y\in V\cap\Z^n$ if
both $x+y$ and $x-y$ are feasible, and shallow otherwise. Relative
to the chosen basis, it is deep if this holds for every basis
vector. The identity
\begin{equation}\label{eq:direction-expansion}
 f(x\pm y)-f(x)=\pm(2y^TQx+c^Ty)+y^TQy
\end{equation}
is the common source of the gradient and curvature bounds.

\begin{lemma}[Deep optimality forces non-negative curvature]\label{lem:deepcurv}
Let $y\in\ker(C)$ and suppose $x^*$ is optimal for~\eqref{eq:iqp2} and
deep with respect to~$y$ (i.e., both $x^*+y$ and $x^*-y$ are feasible).
Then $y^T Qy\ge 0$.
\end{lemma}

\begin{proof}
Since $x^*$ is optimal and both $x^*\pm y$ are feasible, we have
$f(x^*+y)\ge f(x^*)$ and $f(x^*-y)\ge f(x^*)$.  Expanding and adding:
\[
  \bigl[f(x^*+y)-f(x^*)\bigr]+\bigl[f(x^*-y)-f(x^*)\bigr]
  = 2\,y^T Qy \;\ge\; 0. \qedhere
\]
\end{proof}

\begin{lemma}[Structural branching rule]\label{lem:trichotomy}
Let $x^*$ be an attained optimum at a node, and let $y_1,\ldots,y_r$
be the basis from Lemma~\ref{lem:adjugate}. If $x^*$ is shallow,
there is a row $a_j\notin\rowsp C$ such that
\[
 0\le b_j-a_j^Tx^*\le nL_A\Delta(C). \tag{S}
\]
If $x^*$ is deep, then for every $i$,
\[
 y_i^TQy_i\ge0,\qquad
 |2y_i^TQx^*+c^Ty_i|\le y_i^TQy_i. \tag{G}
\]
Finally, if $y_i^TQ\in\rowsp C$ for every $i$, then $Q|_V=0$
and the objective on $Cx=d$ is affine. \hfill(P)
\end{lemma}
\begin{proof}
For (S), one of $x^*\pm y_i$ violates a row $a_j$. Thus
$0\le b_j-a_j^Tx^*<|a_j^Ty_i|\le nL_A\Delta(C)$.
Since $a_j^Ty_i\ne0$ and $Cy_i=0$, this row is outside
$\rowsp C$. For (G), both expressions in
\eqref{eq:direction-expansion} are nonnegative; adding them gives
the curvature condition, and combining them gives the gradient
interval. For (P), expand $v\in V$ in the basis; each row
$y_i^TQ$ annihilates $V$, so $v^TQw=0$ for every $v,w\in V$.
\end{proof}
These are the shallow/deep alternatives in Lokshtanov's framework
\cite{Lokshtanov17}, with the smaller basis bound and an explicit
affine leaf rule. The row $a_j^Tx=b'_j$ fixes a slack level; it
need not be the original boundary $a_j^Tx=b_j$.

\subsection{Affine leaves and determinant growth}
\begin{corollary}[Residual linear objective]\label{cor:linear}
Let $C'x=d'$ be consistent, let $W=\ker(C')$, and suppose $Q|_W=0$.
Choose any rational solution $x_0$ of $C'x=d'$ and set
$\ell=2Qx_0+c$. On this affine set,
\[
 f(x)=\ell^Tx-x_0^TQx_0.
\]
Consequently the residual integer or mixed-integer problem is solved
by minimizing $\ell^Tx$ over its original residual feasible set.
The coefficient vector and constant have polynomial encoding length
in the node data and can be computed by rational linear algebra.
\end{corollary}
\begin{proof}
Write $x=x_0+w$, $w\in W$. Expanding the objective and using
$w^TQw=0$ gives
$f(x)=f(x_0)+(2Qx_0+c)^Tw=\ell^Tx-x_0^TQx_0$.
\end{proof}

The following estimate will also be used for projected constraints
in the mixed algorithm and for the sequential comparison.

\begin{lemma}[Subdeterminant growth after branching]\label{lem:growth}
Put $\Delta:=\Delta(C)$. When adding a row to $C$, with $y$ an adjugate basis vector:
\begin{itemize}
  \item[(a)] Constraint branch (adding $a_j$ to $C$): $\Delta(C')\le nL_A\cdot\Delta$.
  \item[(b)] Gradient branch (adding $2y^T Q$): $\Delta(C')\le 2n^2 L_Q\cdot\Delta^2$.
\end{itemize}
\end{lemma}

\begin{proof}
Let $C' \in \Z^{(k+1) \times n}$ be the matrix obtained by appending the new
row~$r^T$ to $C \in \Z^{k \times n}$.  Consider any $(k'+1) \times (k'+1)$
submatrix $M$ of $C'$.  If $M$ does not use the new row, then
$|\det(M)| \le \Delta$ by definition.  If $M$ does use the new row, we expand
$\det(M)$ along that row.  There are at most $k'+1 \le n$ cofactors, each of
which is the determinant of a $k' \times k'$ submatrix of $C$ and hence
bounded in absolute value by $\Delta$.  Therefore
\[
  |\det(M)| \;\le\; (k'+1) \cdot \max_i |r_i| \cdot \Delta
             \;\le\; n \cdot \max_i |r_i| \cdot \Delta.
\]

For part~(a), the new row is a constraint row $a_j^T$ with
$\|a_j\|_\infty \le L_A$, so $\max_i |r_i| \le L_A$ and
$\Delta(C') \le n L_A \cdot \Delta$.

For part~(b), the new row is the gradient row $h^T = 2y^T Q$, where
$y \in \ker(C)$ is an adjugate basis vector with
$\|y\|_\infty \le \Delta$ (Lemma~\ref{lem:adjugate}).  Each entry satisfies
$|h_\ell| = |2 \sum_{j=1}^n y_j Q_{j\ell}| \le 2n \cdot \Delta \cdot L_Q$,
so $\max_i |r_i| \le 2nL_Q\Delta$ and
$\Delta(C') \le n \cdot 2nL_Q\Delta \cdot \Delta = 2n^2 L_Q \cdot \Delta^2$.
\end{proof}

\section{Curvature batching for IQP}\label{sec:batch}
We prove Theorem~\ref{thm:main} using the framework of
Section~\ref{sec:prelim}. Constraint branches retain an optimum
near a boundary. If an optimum is deep in every basis direction,
all its gradient values lie in bounded integer intervals, and we
can impose the corresponding equations simultaneously. The resulting
quadratic form vanishes on the new kernel, so every path ends
with an ILP after at most one batch.

\subsection{The curvature-batch algorithm}

We classify the basis directions according to the sign of their
quadratic curvature.  This classification determines which type of branching
step to take at each recursion node.

\begin{definition}[Curvature classification]\label{def:curvclass}
Given a basis of integer vectors $\{y_1,\ldots,y_r\}$ of $V=\ker(C)$, we
partition the index set $[r]$ into three \emph{curvature classes}:
\begin{align*}
  \calN&=\{i\in[r]:y_i^T Qy_i<0,\; y_i^T Q\notin\rowsp(C)\}, & \text{(negative curvature)}\\
  \calG&=\{i\in[r]:y_i^T Qy_i\ge 0,\;y_i^T Q\notin\rowsp(C)\}, & \text{(non-negative curvature)}\\
  \calF&=\{i\in[r]:y_i^T Q\in\rowsp(C)\}. & \text{(flat on the current kernel)}
\end{align*}
\end{definition}

\begin{corollary}[Negative curvature forces shallowness]\label{lem:negshallow}
If $y_i^T Qy_i<0$, then every optimal solution is shallow with respect
to~$y_i$, and no gradient branching is required.
\end{corollary}

\begin{proof}
By Lemma~\ref{lem:deepcurv}, any deep optimum would require
$y_i^T Qy_i\ge 0$. Gradient branches exist only to catch deep optima, so none is needed when $y_i^T Qy_i<0$ for some basis vector $y_i$: every optimum is then shallow with respect to $y_i$, and the constraint branches cover it.
\end{proof}

We now present the algorithm. Every returned finite value carries a
corresponding feasible point, and taking a minimum retains that point.
At each recursion node it maintains
an equality system $Cx=d$ with $C\in\Z^{k\times n}$, $\operatorname{rank}(C)=k$,
and kernel $V=\ker(C)$ of dimension $r=n-k$.

\begin{algorithm}[H]
\caption{CurvatureBatch}\label{alg:batch}

\small
\textbf{Input:} $Q,c,A,b$ and independent equalities $Cx=d$.
\textbf{Precondition:} The integer objective is bounded below on the
feasible set (which may be empty). First apply
Theorem~\ref{thm:unb-main} to classify the original instance.
\textbf{Output:} The optimum on this node, or $+\infty$ if infeasible.
\begin{enumerate}[leftmargin=*,itemsep=3pt]
\item Discard an inconsistent affine system. Put $V=\ker C$ and
$r=\dim V$. If $r=0$, evaluate the unique solution only after
checking $Ax\le b$ and integrality; return.
\item Compute the adjugate basis $y_1,\ldots,y_r$ and the classes
$\calN,\calG,\calF$ of Definition~\ref{def:curvclass}.
\item If $\calF=[r]$, solve the residual ILP with the objective
of Corollary~\ref{cor:linear}; return its quadratic value, or
$+\infty$ if infeasible. No further children are generated.
\item Initialize $\mathit{best}=+\infty$. For each
$a_j\notin\rowsp C$ and integer
$b_j'\in[b_j-nL_A\Delta(C),b_j]$, recurse after appending
$a_j^Tx=b_j'$; update $\mathit{best}$.
\item If $\calN=\emptyset$ and $\calG\ne\emptyset$, choose
$\calG'\subseteq\calG$ maximal so that its rows $y_i^TQ$ are
independent modulo $\rowsp C$. For each integer tuple $(z_i)$
with $|z_i|\le y_i^TQy_i$, $i\in\calG'$, append all equations
$2y_i^TQx=z_i-c^Ty_i$ simultaneously and recurse. Update
$\mathit{best}$ with the result.
\item Return $\mathit{best}$.
\end{enumerate}

\end{algorithm}

The batch branch is the essential change. When $\calN=\emptyset$ and
$\calG\ne\emptyset$, the algorithm selects a maximal linearly independent
subset $\calG'\subseteq\calG$ (which can be computed by greedy rank-extension of the row
space of~$C$) and adds all gradient rows $2y_i^T Qx+c^T y_i=z_i$,
$i\in\calG'$, simultaneously for each feasible tuple~$(z_i)$.  The other children fix shallow slack levels; a flat node is solved
using Corollary~\ref{cor:linear}.

\subsection{Correctness and post-batch structure}
\label{sec:spanning}

We first establish correctness, then show that the batch eliminates all
quadratic curvature from the residual problem.

\begin{proposition}[Correctness]\label{prop:correct}
For an instance with a finite attained optimum,
Algorithm~\ref{alg:batch} returns that optimum. For an infeasible
instance it returns $+\infty$.
\end{proposition}

\begin{proof}
Every leaf evaluates a feasible point, so an infeasible instance returns
$+\infty$. Each nonterminal step increases the rank of $C$, so every
path terminates after at most $n$ such steps. For a feasible instance
with an attained optimum, we argue that an optimal solution $x^*$
is feasible in at least one child subproblem at each node, and is therefore
found at some leaf.

If $x^*$ is shallow, some constraint $a_j^T x^* = b_j'$ holds with $b_j'$ in
the enumerated range, so the algorithm creates a child in which $x^*$ is
feasible. Otherwise, we assume $x^*$ is deep. By Corollary~\ref{lem:negshallow}, we have that $\calN=\emptyset$. If $\calG\ne\emptyset$, we construct $\calG'$ by greedily adding indices from $\calG$ whenever $y_i^T Q$
is linearly independent of the current row space. For each $i\in\calG'$ the gradient
value $z_i^* := 2y_i^T Qx^* + c^T y_i$ is an integer with
$|z_i^*|\le y_i^T Qy_i$ (Lemma~\ref{lem:trichotomy}(\textbf{G})).  The algorithm
enumerates all tuples $(z_i)_{i\in\calG'}$ in this range, so in particular it
tries $(z_i^*)_{i\in\calG'}$.  Since $x^*$ already satisfies each equality
$2y_i^T Qx = z_i^*-c^Ty_i$ by construction, $x^*$ is feasible in the corresponding
child. Finally, if $y_i^T Q\in\rowsp(C)$ for every~$i$, the node is a flat ILP
and $x^*$ is found directly. In each case, the recursion continues on a subproblem that still
contains~$x^*$.
\end{proof}

\medskip
\noindent\textbf{Post-batch structure.}
With correctness in hand, we turn to the structural consequence of the batch.
Our claim is that after a single batch step, the quadratic form $Q$
vanishes entirely on the new kernel $W=\ker(C')$, so the residual problem is
linear.  This property is the reason why the algorithm does not batch twice in any root-to-leaf path.

The argument proceeds in two steps.  First, we establish that $v^T Q\in\rowsp(C')$ for every $v\in V$
(Lemma~\ref{lem:gradspan}).  From this, the vanishing of the
restricted form follows immediately (Theorem~\ref{thm:vanish}).

Recall that the batch augments $C$ by appending
the matrix $H$ whose rows are $\{2y_i^T Q\}_{i\in\calG'}$, so
$C'=\bigl[\begin{smallmatrix}C\\H\end{smallmatrix}\bigr]$.  Therefore
$\rowsp(C')=\rowsp(C)+\Span\{y_i^T Q:i\in\calG'\}$. 

\begin{lemma}\label{lem:gradspan}
Suppose $\calN=\emptyset$.  Then for every $v\in V$,
$v^T Q\in\rowsp(C')$.
\end{lemma}

\begin{proof}
Because $\{y_1,\ldots,y_r\}$ span $V$ over $\R$ (Lemma~\ref{lem:adjugate}),
for any $v \in V$, we can write $v=\sum_{i=1}^r\alpha_i y_i$, so $v^T Q=\sum_{i=1}^r\alpha_i y_i^T Q$.
We redistribute each term in the sum.  If $y_i^T Q\in\rowsp(C)$, the term
$\alpha_i y_i^T Q$ is absorbed into~$\rowsp(C)\subseteq\rowsp(C')$.
Since $\calN=\emptyset$,
every remaining index lies in~$\calG$.  For $i\in\calG\setminus\calG'$,
maximality of $\calG'$ implies $y_i^T Q$ lies in
$\rowsp(C)+\Span\{y_j^T Q : j\in\calG'\}=\rowsp(C')$, so its contribution is
again absorbed.  Therefore $v^T Q\in\rowsp(C')$.
\end{proof}

Using Lemma \ref{lem:gradspan}, we show that the remaining problem is linear.

\begin{theorem}\label{thm:vanish}
Let $W = \ker(C')$, where $C'$ is the resulting equality matrix after the batch of gradient constraints is added. Then $Q|_W=0$; that is, $w_1^T Qw_2=0$ for
all $w_1,w_2\in W$.
\end{theorem}

\begin{proof}
Let $v\in V$ and $w\in W$.  By Lemma~\ref{lem:gradspan},
$v^T Q\in\rowsp(C')$, so $v^T Qw=0$ since $w\in\ker(C')$.
Restricting to $v\in W\subseteq V$ gives $Q|_W=0$.
\end{proof}

The residual objective is therefore the affine function in
Corollary~\ref{cor:linear}.

\begin{corollary}[Single batch suffices]\label{prop:onebatch}
The curvature batch is performed at most once along any root-to-leaf path. Every path consists of constraint steps, optionally one gradient batch,
and a terminal ILP or a single-point check.
\end{corollary}

\emph{The batch eliminates all eigenvalues.} \label{rmk:kills-all}
The condition $\calN=\emptyset$ does \emph{not} imply that $Q|_V$ is positive
semidefinite.  It only says that every adjugate basis vector~$y_i$ has
$y_i^T Qy_i\ge 0$. As a concrete example, suppose
$B_V^T QB_V = \bigl(\begin{smallmatrix}1&2\\2&1\end{smallmatrix}\bigr)$. It has
eigenvalues $3$ and $-1$ but the curvatures of each basis vector are positive
($y_1^T Qy_1=y_2^T Qy_2=1>0$), so $\calN=\emptyset$, yet the form has a
negative eigenvalue.  Lemma~\ref{lem:gradspan} shows that regardless of the sign of the eigenvalues, the quadratic form is identically zero on the remaining kernel.

\subsection{Running time}
It remains to prove the running-time bound.

Equipped with Corollary~\ref{prop:onebatch} establishing that only one batch occurs
per root-to-leaf path, we can now prove the main theorem. The key point is that during the constraint phase $\Delta$ does not grow at all, since $C$ consists of rows of $A$ only. After the batch is added, $\Delta$ does grow, but this no longer matters: the remaining problem is an ILP whose parametric factor depends on $n$ alone.

\begin{proof}[Proof of Theorem~\ref{thm:main}]
Throughout this proof set $\bar m=\max(1,m)$.
First remove duplicate coefficient rows, retaining the smallest
right-hand side, and run the classification test of
Theorem~\ref{thm:unb-main}. Its cost is absorbed by the bound
below because $m\le(2L_A+1)^n$ after this preprocessing.
Infeasible and unbounded instances are returned immediately.
A remaining instance has an attained optimum, since $f$ is
integer-valued on its nonempty integer feasible set. The
branching proof needs this attained optimum, not boundedness of
the feasible polyhedron, so no bounding-box replacement is
needed for the pure algorithm.

We analyze the worst-case branching tree of Algorithm~\ref{alg:batch}.  By Corollary~\ref{prop:onebatch}, any root-to-leaf path passes through three phases: at most $s \le n$ constraint rows, at most one batch of gradient rows, and an ILP leaf.  We sum over the at most $n+1$ possible positions of the batch and
multiply the branching factors for each position. This additional
polynomial factor, and the factor $n+1$ converting leaves to nodes,
are absorbed below.

During the constraint phase, the matrix $C$ consists entirely of rows of~$A$, so every minor of~$C$ is a minor of~$A$ and $\Delta(C) \le \Delta(A)$ throughout.  At each constraint step the algorithm branches over which row $a_j\notin\rowsp(C)$ to add (at most $m$ choices) and which value $b_j'\in\{b_j-nL_A\Delta(A),\ldots,b_j\}$ to assign (at most $nL_A\Delta(A)+1$ choices), giving a branching factor of at most $\bar m\cdot(nL_A\Delta(A)+1)$ per step.  Over at most $n$ constraint steps, the combined branching factor is at most $\bigl[\bar m\cdot(nL_A\Delta(A)+1)\bigr]^n$.

At the batch node, the subdeterminant still satisfies $\Delta(C) \le \Delta(A)$, so each adjugate basis vector has $\|y_i\|_\infty \le \Delta(A)$ and the curvature along each direction satisfies $y_i^T Qy_i \le n^2 L_Q \Delta(A)^2$.  For each $i \in \calG'$, the algorithm enumerates at most $2y_i^T Qy_i + 1$
integer values, so the total number of children at the batch node is at most
$\prod_{i\in\calG'}(2\,y_i^TQy_i+1) \le (3n^2 L_Q\Delta(A)^2)^{n}$.  Each leaf is either a checked single point or an ILP (Corollary~\ref{cor:linear}), solved by fixed-dimension integer linear programming in $n^{O(n)}\cdot\mathrm{poly}(\varphi)$ time. Combining the three contributions, we obtain for the running time:
\[
  \bigl(\bar m\,(nL_A\Delta(A)+1)\bigr)^n
  \;\cdot\;
  (3n^2 L_Q\Delta(A)^2)^{n}
  \;\cdot\;
  n^{O(n)}
  \;\cdot\;
  \mathrm{poly}(\varphi)
    \le (\bar m n\Delta(A)L_A L_Q)^{O(n)} \;\cdot\;
  \mathrm{poly}(\varphi)
\]
There are at most $(2L_A+1)^n$ distinct coefficient rows.
Keeping the tightest right-hand side for each row therefore gives
\begin{equation}\label{eq:complexitybound}
(n\,\Delta(A)\,(2L_A+1)^n\,L_Q)^{O(n)}
\cdot \mathrm{poly}(\varphi).
\end{equation}

Finally, utilizing $\Delta(A)\le (n L_A)^n\le (nL)^n$ and $L = \max\{L_A,L_Q\}$, we get the bound in Theorem \ref{thm:main} in $n$ and $L$ alone:
\begin{equation}\label{eq:nlcomplexitybound}
      (nL)^{O(n^2)}\;\cdot\;\mathrm{poly}(\varphi).
\end{equation}
\end{proof}

\begin{proof}[Proof of Corollary~\ref{cor:tu}]
If $A$ is TU, then $L_A = \Delta(A) =1$. Use the bound before duplicate-row removal with $L_A=\Delta(A)=1$
to obtain $(\max(1,m)nL_Q)^{O(n)}\mathrm{poly}(\varphi)$. There are at most
$3^n$ distinct TU rows, yielding
$2^{O(n^2)}(nL_Q)^{O(n)}\mathrm{poly}(\varphi)$.
\end{proof}
\section{Curvature batching for bounded MIQP}
\label{sec:mixed}

Write $x=(x_I,x_C)$ with $x_I\in\Z^{p}$, $x_C\in\R^{q}$, $n=p+q$, and
consider the node form of the mixed problem~\eqref{eq:miqp0},
\begin{equation}\label{eq:miqp}
  \min\bigl\{f(x)=x^TQx+c^Tx \;:\; Ax\le b,\ Cx=d,\
  x\in\ZR\bigr\},
\end{equation}
with integer original data as in~\eqref{eq:iqp2}. Initially $C$ is
empty; any user-supplied equalities are first included as pairs of
opposite inequalities in $Ax\le b$, and hence in $L_A$ and
$\Delta(A)$. Later $C$ is the integer matrix of independent rows
generated by the algorithm, and $d$ may be rational with the common
denominator specified below. We assume throughout that
$P=\{x:Ax\le b\}$ is a polytope; the feasible set is then a finite
union of compact fibers and the minimum is attained
(unboundedness is treated in Section~\ref{sec:unbounded}). We make
\emph{no convexity assumption on any block of $Q$}.

\begin{theorem}[Mixed curvature batch]\label{thm:m-main}
MIQP~\eqref{eq:miqp0} over a polytope, with initial equalities incorporated as above, can be solved in time
$(nL)^{O(n^{2}(q+1)^{2})}\cdot\mathrm{poly}(\varphi)$. For fixed $q$ this
is $(nL)^{O(n^{2})}\cdot\mathrm{poly}(\varphi)$, and at $q=0$ the
algorithm reduces to the pure case with the same coarse
$(nL)^{O(n^2)}$ bound as Theorem~\ref{thm:main}.
\end{theorem}

Three features distinguish this problem from pure IQP. Continuous
movements lead to exact stationarity equations, branch values may be
fractional, and an integer displacement may require a compensating
continuous movement to preserve feasibility. We address these in order:
a continuous stationarity batch, a common denominator for later branch
values, and a projected family of inequalities describing feasible cosets.
The recursion continues to use the joint variables $(x_I,x_C)$.

\subsection{The continuous phase}
At a node $Cx=d$, let $V=\ker(C)$ and define the \emph{continuous
subkernel}
\[
U \;:=\; V\cap\bigl(\{0\}^{p}\times\R^{q}\bigr)
 \;=\; \ker\!\begin{pmatrix} C\\ E_I\end{pmatrix},
\qquad
E_I := \begin{pmatrix} I_p & 0\end{pmatrix}\in\Z^{p\times n},
\]
of dimension $q'\le q$. Steps along integer vectors $y\in V\cap\Z^{n}$
and along $tu$, $u\in U$, $t\in\R$, both preserve $\ZR$. We adopt the
convention $\Delta(\emptyset)=1$ for an empty matrix.

\begin{lemma}[Augmented subdeterminants; $U$-basis]\label{lem:m-augdelta}
$\Delta\bigl(\begin{smallmatrix}C\\E_I\end{smallmatrix}\bigr)
\le\Delta(C)$. Further, $U$ admits a basis of integer vectors
$u_1,\dots,u_{q'}$ with
$\|u_j\|_\infty\le q^{q/2}E^{q}$ whenever all entries of $C$ are
bounded by $E$.
\end{lemma}

\begin{proof}
For the first claim, expand any square submatrix of
$\binom{C}{E_I}$ by Laplace along its unit rows: each such row either
makes the determinant zero (if its unit column is absent) or reduces
to a smaller submatrix of $C$; iterate. If $q=0$, the basis is empty and the second assertion is
vacuous. For $q\ge1$, $u\in U$ iff
$u_I=0$ and $u_C\in\ker(C_{\cdot C})$, where
$C_{\cdot C}\in\Z^{k\times q}$ is the block of continuous columns of
$C$. Apply Lemma~\ref{lem:adjugate} to a \emph{row basis} of
$C_{\cdot C}$ (same kernel) and pad with zeros; the entries are minors
of $C_{\cdot C}$ of order at most $q$, bounded by Hadamard's
inequality.
\end{proof}
The branching dichotomy has the following continuous analogue.

A feasible $x$ is \emph{$\R$-deep} with respect to $u\in U$ if
$x\pm\varepsilon u$ is feasible for some $\varepsilon>0$;
equivalently, no row of $A$ tight at $x$ has $a_\ell^Tu\neq0$.

\begin{lemma}[Continuous dichotomy]\label{lem:m-dichotomy}
Let $x^{*}$ be optimal for~\eqref{eq:miqp} and $u\in U$. Then exactly
one of:
\begin{itemize}
\item[\emph{(S$_\R$)}] \emph{Blocked}: some row $a_\ell$ with
$a_\ell^Tu\neq 0$ satisfies $a_\ell^Tx^{*}=b_\ell$; necessarily
$a_\ell\notin\rowsp(C)$.
\item[\emph{(G$_\R$)}] \emph{Stationary}: $x^{*}$ is $\R$-deep with
respect to $u$; then $2u^TQx^{*}+c^Tu=0$ and $u^TQu\ge0$.
\end{itemize}
\end{lemma}

\begin{proof}
If $x^{*}$ is not $\R$-deep with respect to $u$, then some tight row
has $a_\ell^Tu\neq0$ (otherwise all tight rows are preserved and all
slack rows remain slack for small $|t|$, making $x^{*}\pm tu$
feasible); tightness means $a_\ell^Tx^{*}=b_\ell$ exactly, and
$a_\ell\in\rowsp(C)$ would force $a_\ell^Tu=0$ since $u\in\ker(C)$.
If $x^{*}$ is $\R$-deep, then
$\phi(t)=f(x^{*}+tu)=f(x^{*})+t(2u^TQx^{*}+c^Tu)+t^{2}u^TQu$ has a
local minimum at $t=0$ on an open interval, forcing $\phi'(0)=0$; and
$u^TQu<0$ together with $\phi'(0)=0$ would give $\phi(t)<\phi(0)$ for
small $t\neq0$, contradicting optimality.
\end{proof}

\begin{corollary}\label{cor:m-allstationary}
If $x^{*}$ is $\R$-deep with respect to every basis vector
$u_1,\dots,u_{q'}$, then no tight row at $x^{*}$ has a nonzero inner
product with any element of $U$; consequently $x^{*}$ is $\R$-deep
with respect to \emph{every} $u\in U$,
\[
2u^TQx^{*}+c^Tu=0 \ \text{ and } \ u^TQu\ge0
\qquad\text{for all } u\in U,
\]
and in particular $Q|_{U}\succeq0$.
\end{corollary}

\begin{proof}
$\R$-deepness with respect to $u_j$ excludes tight rows with
$a_\ell^Tu_j\neq0$; over all $j$, tight rows vanish on
$\Span\{u_j\}=U$, giving $\R$-deepness along every $u\in U$.
Lemma~\ref{lem:m-dichotomy}(G$_\R$) applies to each $u$; alternatively
stationarity for all $u$ follows from the basis case by linearity in
$u$.
\end{proof}

Thus if $Q|_{U}\not\succeq0$, decidable by a square-root-free
$LDL^{T}$ factorization of $B_U^TQB_U$, then every optimal solution
is blocked along some continuous direction and only (S$_\R$) branches
are required: the continuous analogue of
Corollary~\ref{lem:negshallow}.
When every optimal solution is stationary along $U$, all
stationarity rows are appended in one step.

\begin{definition}[Stationarity batch]\label{def:m-cs}
Let $u_1,\dots,u_{q'}$ be the basis of Lemma~\ref{lem:m-augdelta}.
First check consistency of $Cx=d$ together with all equations
$2u_j^TQx=-c^Tu_j$, $j\in[q']$; if inconsistent, discard this child.
Otherwise choose $J\subseteq[q']$ maximal such that
$\{u_j^TQ\}_{j\in J}$ is linearly independent modulo $\rowsp(C)$,
and append to $Cx=d$ only the rows
\[
2u_j^TQ\,x \;=\; -\,c^Tu_j, \qquad j\in J .
\]
\end{definition}

All $|J|\le q'$ rows are appended in one step; this single Laplace
expansion, rather than an iterated per-direction recursion, is what
prevents repeated squaring of $\Delta$ (see the proof of
Theorem~\ref{thm:m-main}).

\begin{lemma}[$U$-flatness]\label{lem:m-flat}
Let $C^{0}x=d^{0}$ denote the system after the batch, with
$V^{0}=\ker(C^{0})$ and $U^{0}=V^{0}\cap(\{0\}^p\times\R^q)$. Then:
\begin{itemize}
\item[(a)] \emph{(Unconditional.)} $u^TQ\in\rowsp(C^{0})$ for every
$u\in U$; hence $u^TQv=0$ for all $u\in U$, $v\in V^{0}$, and
$Q|_{U^{0}}=0$.
\item[(b)] \emph{(Every retained batch.)}
$2u^TQx+c^Tu\equiv0$ on the affine set for every $u\in U$, and $f$ is
$U^{0}$-invariant there: $x,x'$ in the affine set with $x_I=x'_I$
satisfy $f(x)=f(x')$.
\end{itemize}
\end{lemma}

\begin{proof}
(a) For $u=\sum_j\beta_ju_j\in U$: rows with $j\in J$ lie in
$\rowsp(C^{0})$ by construction; for $j\notin J$, maximality gives
$u_j^TQ\in\rowsp(C)+\Span\{u_i^TQ:i\in J\}\subseteq\rowsp(C^{0})$.
Membership in $\rowsp(C^{0})$ gives $u^{T}Qv=0$ for every
$v\in\ker(C^{0})$, and
$U^{0}\subseteq V^{0}$ gives $Q|_{U^{0}}=0$.

(b) The consistency check supplies a point $\bar{x}$ satisfying all
stationarity equations and $Cx=d$, hence also $C^0x=d^0$.
By (a), each functional $x\mapsto 2u^TQx$ is constant on the
affine set; its constant value is its value at $\bar{x}$, namely
$-c^Tu$ (by linearity in $u$). For the invariance:
$x-x'\in\ker(C^{0})$ with zero integer part means $x-x'\in U^{0}$,
and
$f(x)-f(x')=(2(x-x')^TQx'+c^T(x-x'))+(x-x')^TQ(x-x')=0+0$ by the
constancy and (a).
\end{proof}

The consistency check preserves the branch containing a fixed optimal
$x^{*}$ that reaches the batch, since $x^{*}$ satisfies every
stationarity equation (Corollary~\ref{cor:m-allstationary}).
It also ensures that (b) holds on every retained batch child.
Retaining an independent row basis preserves the full affine solution
set because the full stationarity system was checked for consistency. We call the node after the batch, or any node with
$U=\{0\}$, the \emph{continuous-flat node} of its path, and denote its data
$(C^{0},d^{0},U^{0})$; this terminology refers only to continuous fiber directions, and
the objective may still be quadratic in the other directions.
The vector $d^{0}$ is \emph{integral}, since the
appended right-hand sides $-c^Tu_j$ are integers and all earlier
right-hand sides are entries of $b$ and $d$.

\subsection{Projection and common-denominator discretization}

The pure algorithm branches on integer values of integer functionals
at $x^{*}$; with $x^{*}_C$ real this fails. The following two lemmas
restore a discrete value set with \emph{one common denominator fixed
at the continuous-flat node}.

\begin{lemma}[All later rows vanish on $U^{0}$]\label{lem:m-vanish}
Every row appended at a descendant of the continuous-flat node, projected
rows $h$ (Lemma~\ref{lem:m-proj}) and gradient rows $2y^TQ$ with
$y\in\ker(C_{\mathrm{cur}})$, vanishes on $U^{0}$. Consequently
$U=U^{0}$ at every descendant, and Lemma~\ref{lem:m-flat}(a)
persists.
\end{lemma}

\begin{proof}
The rows $h$ vanish on $U^{0}$ by construction
(Lemma~\ref{lem:m-proj}(ii)). For a gradient row and $u\in U^{0}$:
$(2y^TQ)u=2\,u^TQy=0$, since $u^TQ\in\rowsp(C^{0})$ by
Lemma~\ref{lem:m-flat}(a) and
$y\in\ker(C_{\mathrm{cur}})\subseteq\ker(C^{0})$. Since appended rows
only enlarge $\rowsp(C)$ while vanishing on $U^{0}$, we get
$U^{0}\subseteq U\subseteq U^{0}$ at every descendant.
\end{proof}

\begin{lemma}[Common-denominator discretization]\label{lem:m-cramer}
At the continuous-flat node select independent rows of $\binom{C^0}{E_I}$
forming a row basis $R$. Choose a nonsingular square column
submatrix $M$ of $R$, and set
$\delta_0:=|\det M|\le\Delta(C^0)$.
If the rank is zero, take the empty matrix and $\delta_0=1$.
Let $\ell\in\Z^{n}$ satisfy $\ell^Tu=0$ for all $u\in U^{0}$, and let
$x^{*}$ satisfy $C^{0}x^{*}=d^{0}$ with $x^{*}_I\in\Z^{p}$. Then
\[
\ell^Tx^{*}\;\in\;\tfrac{1}{\delta_0}\,\Z .
\]
In particular this applies, with the \emph{same} $\delta_0$, to every
projected row $h$ and every gradient functional $2y^TQ$ arising at any
descendant of the continuous-flat node, evaluated at any optimal solution on
that branch.
\end{lemma}

\begin{proof}
The solution set of $\binom{C^{0}}{E_I}x=\binom{d^{0}}{x^{*}_I}$ is
$x^{*}+U^{0}$, and $\ell$ is constant on it, so
$\ell^Tx^{*}=\ell^T\hat{x}$ for the solution $\hat{x}$ obtained
by retaining the selected independent equations, setting nonbasic
coordinates to zero, and applying Cramer's rule to $M$.
The omitted equations hold as well: the full system is consistent
(it contains $x^*$), and its omitted rows are linear combinations
of the selected rows with the same combinations on the right-hand
sides. The right-hand side $\binom{d^{0}}{x^{*}_I}$ is
integral (Lemma~\ref{lem:m-flat} discussion; $x^{*}_I\in\Z^p$), so
$\hat{x}\in\tfrac{1}{\delta_0}\Z^{n}$ and
$\ell^T\hat{x}\in\tfrac{1}{\delta_0}\Z$. The bound
$\delta_0\le\Delta(C^{0})$ is Lemma~\ref{lem:m-augdelta}. The "in
particular" clause combines Lemma~\ref{lem:m-vanish} (the
functionals vanish on $U^{0}$) with the fact that any point of a
descendant's affine set satisfies $C^{0}x=d^{0}$.
\end{proof}

Accordingly, descendants of the continuous-flat node carry an \emph{integer}
constraint matrix and a rational right-hand-side vector all of whose
entries lie in $\tfrac{1}{\delta_0}\Z$; appended rows are never
scaled, so $\Delta(C)$ is always computed on integer data and grows
only through the entries of the appended rows.
With the denominator fixed, integer-direction branching proceeds at
the descendants after the continuous phase.

At a descendant of the continuous-flat node, on the branch of a fixed optimal
$x^{*}$, call $x^{*}$ \emph{$I$-deep} with respect to an integer basis vector $y$ of $V$ (Lemma~\ref{lem:adjugate}) if both cosets
$x^{*}\pm y+U^{0}$ contain points satisfying $Ax\le b$, and
\emph{$I$-shallow} otherwise. (Such points automatically satisfy all
original equalities and mixed-integrality, since
$y\in\ker(C_{\mathrm{cur}})$, $u_I=0$, $y_I\in\Z^{p}$.)

\begin{lemma}[Exact gradient identity on cosets]\label{lem:m-gradid}
On $x^{*}$'s branch, for any $y\in V\cap\Z^{n}$ and any $u\in U^{0}$,
\[
f(x^{*}+y+u)-f(x^{*})
= \bigl(2y^TQx^{*}+c^Ty\bigr) + y^TQy ,
\]
independently of $u$. Hence if $x^{*}$ is $I$-deep with respect to
$y$, then $y^TQy\ge0$ and
\[
\bigl|\,2y^TQx^{*}+c^Ty\,\bigr| \le y^TQy,
\qquad
2y^TQx^{*}+c^Ty \in \tfrac{1}{\delta_0}\Z .
\]
If instead $y^TQy<0$, then $x^{*}$ is $I$-shallow with respect to
$y$.
\end{lemma}

\begin{proof}
Expand $f(x^{*}+y+u)-f(x^{*})$ and use, on $x^{*}$'s branch:
$2u^TQx^{*}+c^Tu=0$ (Lemma~\ref{lem:m-flat}(b) with
Corollary~\ref{cor:m-allstationary}), $y^TQu=u^TQy=0$
(Lemma~\ref{lem:m-flat}(a) with
$y\in\ker(C_{\mathrm{cur}})\subseteq\ker(C^{0})$), and $u^TQu=0$.
The identity is independent of $u$, so the two coset witnesses
$u_{+},u_{-}$ (possibly different) give
$\pm(2y^TQx^{*}+c^Ty)+y^TQy\ge0$; adding yields $y^TQy\ge0$ and
combining yields the sandwich. Membership in $\tfrac1{\delta_0}\Z$ is
Lemma~\ref{lem:m-cramer} with $\ell=2Qy$ (note $c^Ty\in\Z$). The
last claim is the contrapositive of $y^TQy\ge0$ under $I$-deepness.
\end{proof}

\emph{Why $U^{0}$-cosets are forced.} One might hope to use cosets of the larger pre-step subkernel $U$;
this fails: for $u\in U\setminus U^{0}$ the identity acquires the
term $u^TQu\ge0$ on the wrong side of the sandwich, and the witness's
curvature is uncontrolled. The batch manufactures exactly the
subspace on which $Q$ vanishes identically
(Lemma~\ref{lem:m-flat}(a)), which is what the deep-point argument
requires.

For $I$-shallow solutions we use valid inequalities for coset
feasibility, obtained by projecting out $U^{0}$ \emph{once, at the
continuous-flat node}:

\begin{lemma}[Projected constraint family]\label{lem:m-proj}
Write $q_0=\dim U^0$. Let $D\in\Z^{k_D\times n}$ collect defining rows for $U^{0}$ (the
continuous-column blocks of $C^{0}$ padded to length $n$), so that
$U^{0}=\{u:u_I=0,\ Du=0\}$, with entries of $D$ bounded by
$(nL)^{O(q+1)}$. Choose a nonnegative integer generator $\lambda$ for each extreme
ray of the pointed cone
$\{\lambda\ge0:\lambda^TAu=0\ \forall u\in U^{0}\}$, using the
Cramer construction below. Let $\calH$ be the family of inequalities
$(h,g_h)=(A^T\lambda,\lambda^Tb)$ indexed by these generators.
Repeated left-hand sides may be retained with their own right-hand
sides. Each generator has $|\mathrm{supp}(\lambda)|\le q_0+1$.
Both $h$ and $g_h$ use the same scaling. Then:
\begin{itemize}
\item[(i)] For $\xi$ with $A\xi\le b$ possibly violated:
the coset $\xi+U^{0}$ contains a point of $P$ iff
$h^T\xi\le g_h:=\lambda^Tb$ for all $h\in \calH$.
\item[(ii)] Every $h\in \calH$ vanishes on $U^{0}$, and
$\|h\|_\infty\le(nL)^{O((q+1)^{2})}$.
\item[(iii)] If, on $x^{*}$'s branch, $x^{*}$ is $I$-shallow with
respect to a basis vector $y$, then some $h\in \calH$ satisfies
\[
h\notin\rowsp(C_{\mathrm{cur}}),\qquad
h^Tx^{*}\in\bigl[g_h-\|h\|_1\|y\|_\infty,\ g_h\bigr]
\cap\tfrac{1}{\delta_0}\Z .
\]
\end{itemize}
\end{lemma}

\begin{proof}
(i) After replacing $D$ by a row basis (same kernel) we may assume
$k_D\le q$. Let $B_U\in\Z^{n\times q_0}$ have columns the basis of
$U^{0}$ from Lemma~\ref{lem:m-augdelta}. The coset $\xi+U^{0}$
contains a point of $P$ iff the linear system
$(AB_U)w\le b-A\xi$ in $w\in\R^{q_0}$ is feasible. By Farkas' lemma,
this system is infeasible iff there exists $\lambda\ge0$ with
$\lambda^TAB_U=0$ and $\lambda^T(b-A\xi)<0$; hence the coset meets
$P$ iff $\lambda^TA\xi\le\lambda^Tb$ for every $\lambda$ in the cone
$\Lambda=\{\lambda\ge0:\lambda^TAB_U=0\}$. The cone $\Lambda$ is
polyhedral and pointed (it is contained in $\R^{m}_{\ge0}$), so it is
generated by its finitely many extreme rays, and the condition
$\lambda^TA\xi\le\lambda^Tb$, being linear in $\lambda$, holds on all
of $\Lambda$ iff it holds on every extreme ray. Since
$\lambda^TAB_U=0$ iff $\lambda^TAu=0$ for all $u\in U^{0}$, the rays
of $\Lambda$ are those in the definition of $\calH$, and (i) follows
with $g_h=\lambda^Tb$. An extreme ray lies on a one-dimensional face
of $\Lambda$, so it satisfies $m-1$ linearly independent constraints
among the equations $\lambda^TAB_U=0$, of rank at most $q_0$, and the
sign constraints $\lambda_i=0$; hence at most $q_0+1$ entries of
$\lambda$ are nonzero.

(ii) If $q=0$, then $\Lambda=\R^m_{\ge0}$; choose its coordinate
generators, so $h$ is an original row of $A$ and the bound is
$\|h\|_\infty\le L_A\le(nL)^{O((q+1)^2)}$. Hence assume
$q\ge1$. Vanishing on $U^{0}$ is the defining condition of $\Lambda$:
$h^Tu=\lambda^TAu=0$. For the entry bound, fix an extreme ray
$\lambda$ with support $S$, $|S|\le q_0+1$. The condition
$\lambda_S^TA_Su=0$ for all $u\in U^{0}$ says that $A_S^T\lambda_S$
lies in $(U^{0})^{\perp}=\rowsp\binom{E_I}{D}$, that is,
$A_S^T\lambda_S=E_I^T\mu+D^T\rho$ for some $\mu\in\R^{p}$,
$\rho\in\R^{k_D}$. The rows of $E_I$ are unit vectors on the integer
coordinates, so the integer block of this equation determines
$\mu=(A_S^T\lambda_S)_I$ outright and imposes no condition on
$\lambda_S$; what remains is the continuous block,
\[
(A_S^T\lambda_S)_C \;=\; D_C^T\rho ,
\]
a system of $q$ equations in the $|S|+k_D\le 2q+1$ unknowns
$(\lambda_S,\rho)$ with coefficients bounded by
$\max(L_A,\|D\|_\infty)=(nL)^{O(q+1)}$. Because $D_C$ has independent rows, $\rho$ is determined uniquely
by $\lambda_S$. The kernel of
$[\,A_{S,C}^{T}\ {-}D_C^{T}\,]$ has dimension one: otherwise a
second kernel vector would perturb the strictly positive
$\lambda_S$ in both directions while keeping the same support,
contradicting extremality. Select independent equations of this
system, giving a matrix with one more column than its rank; the standard Cramer description of such a
kernel vector has entries equal, up to sign, to maximal minors of
that submatrix, of order at most $q$, hence of absolute value at most
$q^{q/2}\bigl((nL)^{O(q+1)}\bigr)^{q}=(nL)^{O((q+1)^{2})}$ by Hadamard's
inequality. Finally
$\|h\|_\infty=\bigl\|\sum_{i\in S}\lambda_ia_i\bigr\|_\infty
\le(q_0+1)(nL)^{O((q+1)^{2})}L_A=(nL)^{O((q+1)^{2})}$. The bound is obtained
against the defining rows $D$, whose entries are $(nL)^{O(q+1)}$;
running the same computation through a basis of $U^{0}$, whose
entries can reach $(nL)^{O((q+1)^{2})}$, would inflate the rays to
$(nL)^{O((q+1)^{3})}$.

(iii) $I$-shallowness of $x^{*}$ with respect to $y$ means that
$x^{*}+y+U^{0}$ or $x^{*}-y+U^{0}$ contains no point of $P$; by (i)
some $h\in \calH$ then satisfies $h^T(x^{*}\pm y)>g_h$, so
$h^Tx^{*}>g_h-|h^Ty|\ge g_h-\|h\|_1\|y\|_\infty$, while
$h^Tx^{*}\le g_h$ because $x^{*}$'s own coset contains the point
$x^{*}\in P$. If $h$ were in $\rowsp(C_{\mathrm{cur}})$ we would have
$h^Ty=0$, as $y\in\ker(C_{\mathrm{cur}})$, hence
$h^T(x^{*}\pm y)=h^Tx^{*}\le g_h$, contradicting the violation; so
$h\notin\rowsp(C_{\mathrm{cur}})$. Membership of $h^Tx^{*}$ in
$\tfrac1{\delta_0}\Z$ is Lemma~\ref{lem:m-cramer} with $\ell=h$.
\end{proof}
\begin{table}[t]
\centering\small
\renewcommand{\arraystretch}{1.2}
\begin{tabular}{P{0.16\textwidth}P{0.36\textwidth}P{0.38\textwidth}}
\toprule
Data & Created when the continuous phase ends & Role at every descendant \\
\midrule
$C^0,d^0$ & Original tight rows and one stationarity batch & Ancestor equations retained by every descendant \\
$U^0$ & Continuous kernel of $C^0$ & Every later row vanishes on this subspace \\
$\delta_0$ & A nonsingular minor of a row basis of $\binom{C^0}{E_I}$ & All enumerated values lie in $\delta_0^{-1}\Z$ \\
$\calH$ & Extreme rays of the nonnegative multiplier cone & Describes feasibility after adjustment within $U^0$ \\
\bottomrule
\end{tabular}
\caption{Data fixed at the end of the continuous phase. Fiber invariance
of the objective holds at every retained continuous-flat node and its descendants.}
\label{tab:m-invariants}
\end{table}

\subsection{Algorithm and correctness}
The pieces assemble into the following recursion.

\begin{algorithm}[H]\caption{MixedCurvatureBatch}
\small\label{alg:m-batch}
A node carries $(C,d)$ and a flag recording whether its continuous
phase has ended. At the root the flag is unset and $C$ is empty.
The algorithm returns $+\infty$ if no child has a feasible leaf;
otherwise its returned value is accompanied by a minimizing feasible point.
\begin{enumerate}
\item Discard an inconsistent affine system. If $\ker C=\{0\}$,
solve $Cx=d$, check $Ax\le b$ and $x_I\in\Z^p$, and evaluate $f$
if these checks pass; return.

\item If the node is still in the continuous phase, compute
$U=\ker C\cap(\{0\}^p\times\R^q)$. If $U=\{0\}$, mark the
node continuous-flat and initialize $(C^0,d^0,U^0,\delta_0,\calH)$ using
Lemmas~\ref{lem:m-cramer} and~\ref{lem:m-proj}; continue to step~3.
Otherwise create the following children, return the best value
among them, and do not perform step~3 at the current node:
\begin{itemize}
\item[(S$_\R$)] For each $a_\ell$ with $a_\ell|_U\ne0$, append
$a_\ell^Tx=b_\ell$; its child remains in the continuous phase.
There are at most $m$ such children, and their appended rows are
outside $\rowsp C$.
\item[(CS)] If $Q|_U\succeq0$, apply the stationarity batch, including
its consistency check (Definition~\ref{def:m-cs}). If consistent,
create one child, mark it continuous-flat, and initialize its
$(C^0,d^0,U^0,\delta_0,\calH)$. This phase transition is made even
when the batch appends no rows.
\end{itemize}

\item At a continuous-flat node, compute the adjugate basis $\{y_i\}_{i=1}^r$
of $\ker C$ and the classes $\calN,\calG,\calF$ of
Definition~\ref{def:curvclass}. If $\calF=[r]$, first apply the
terminal rule:
\begin{itemize}
\item[(P)] Choose a rational $x_0$ with $Cx_0=d$ and minimize
$(2Qx_0+c)^Tx$ over $Ax\le b$, $Cx=d$, $x_I\in\Z^p$ by a
fixed-integer-dimension MILP algorithm. Return its actual quadratic
objective value using Corollary~\ref{cor:linear}; if infeasible,
return $+\infty$.
\end{itemize}
If the terminal rule does not apply, create the following children,
all retaining the same flat-node data, and return their best value:
\begin{itemize}
\item[(S$_\Z$)] For each $(h,g_h)\in \calH$ with
$h\notin\rowsp C$ and each
\[
 v\in[g_h-\|h\|_1\Delta(C),g_h]\cap\tfrac1{\delta_0}\Z,
\]
append the unscaled integer row $h^Tx=v$.
\item[(Batch$_\Z$)] If $\calN=\emptyset$ and
$\calG\ne\emptyset$, choose $\calG'$ maximal as in
Algorithm~\ref{alg:batch}. For every tuple
$z_i\in\tfrac1{\delta_0}\Z$ with $|z_i|\le y_i^TQy_i$,
append $2y_i^TQx=z_i-c^Ty_i$ for all $i\in\calG'$.
\end{itemize}
\end{enumerate}
\end{algorithm}

\begin{proposition}[Correctness]\label{prop:m-correct}
If the original instance is feasible, some leaf of
MixedCurvatureBatch evaluates an optimal solution of~\eqref{eq:miqp}.
If it is infeasible, the algorithm returns $+\infty$.
\end{proposition}

\begin{proof}
Every leaf evaluates a feasible point of~\eqref{eq:miqp}: a leaf of
type~1 checks $Ax\le b$ and $x_I\in\Z^{p}$ before evaluating, and a
(P)-leaf optimizes over the residual mixed set, which is contained in
the feasible set. Hence every returned value is the value of a
feasible point, and it suffices to show that some leaf evaluates an
optimal solution.

Fix an optimal $x^{*}$ and follow it down the tree: we show that at
every node whose affine set contains $x^{*}$, some child's affine set
again contains $x^{*}$, and that this descent terminates at a leaf
evaluating $x^{*}$ or a point of equal value.

\emph{Continuous phase.} Let the node have $U\neq\{0\}$ with basis
$u_1,\dots,u_{q'}$ (Lemma~\ref{lem:m-augdelta}), and apply
Lemma~\ref{lem:m-dichotomy} to each $u_j$. If some $u_j$ is blocked,
the blocking row satisfies $a_\ell^Tx^{*}=b_\ell$,
$a_\ell^Tu_j\neq0$ (so $a_\ell|_U\neq0$), and
$a_\ell\notin\rowsp(C)$; the corresponding (S$_\R$) child appends
exactly the equality that $x^{*}$ satisfies, so its affine set
contains $x^{*}$, and $\dim U$ drops by one because the appended row
has nonzero restriction to $U$. If no basis vector is blocked,
Corollary~\ref{cor:m-allstationary} gives $Q|_{U}\succeq0$, so the
(CS) child is created, and it also gives the stationarity identities
$2u_j^TQx^{*}+c^Tu_j=0$ for every $j\in[q']$, so $x^{*}$ satisfies
every stationarity equation. The consistency check therefore passes,
and the (CS) child contains $x^{*}$. Since each
(S$_\R$) step reduces $\dim U$ and the (CS) step, or $U=\{0\}$, ends
the phase, the path of $x^{*}$ reaches its continuous-flat node after at most
$q$ (S$_\R$)-steps.

\emph{Descendants after the continuous phase.} By Lemma~\ref{lem:m-vanish}, $U=U^{0}$ at
every descendant, and the branch data $(\delta_0,\calH)$ are fixed at the
continuous-flat node. Let $\{y_i\}$ be the adjugate basis of
$\ker(C_{\mathrm{cur}})$ (Lemma~\ref{lem:adjugate}; the constraint
matrix is integral by the discussion after
Lemma~\ref{lem:m-cramer}). Two cases.

If $x^{*}$ is $I$-shallow with respect to some $y_i$, then
Lemma~\ref{lem:m-proj}(iii) supplies
$h\in \calH\setminus\rowsp(C_{\mathrm{cur}})$ with
$h^Tx^{*}\in[g_h-\|h\|_1\|y_i\|_\infty,\,g_h]\cap\tfrac1{\delta_0}\Z$.
Since $\|y_i\|_\infty\le\Delta(C_{\mathrm{cur}})$, the value
$h^Tx^{*}$ is among those enumerated by (S$_\Z$), and the
corresponding child contains $x^{*}$.

If $x^{*}$ is $I$-deep with respect to every $y_i$, then
Lemma~\ref{lem:m-gradid} gives $y_i^TQy_i\ge0$ for all $i$, so
$\calN=\emptyset$. If $\calG\neq\emptyset$, the (Batch$_\Z$) branch
is created, and the tuple $z_i^{*}=2y_i^TQx^{*}+c^Ty_i$,
$i\in\calG'$, lies in $\tfrac1{\delta_0}\Z$ with
$|z_i^{*}|\le y_i^TQy_i$ (Lemma~\ref{lem:m-gradid}), so it is
enumerated and the corresponding child contains $x^{*}$; for
$i\in\calG\setminus\calG'$, maximality of $\calG'$ places $y_i^TQ$
in $\rowsp(C)+\Span\{y_j^TQ:j\in\calG'\}$, exactly as in
Proposition~\ref{prop:correct}. After the batch,
$v^TQ\in\rowsp(C_{\mathrm{new}})$ for every
$v\in\ker(C_{\mathrm{cur}})$: for the spanning argument this is
Lemma~\ref{lem:gradspan} verbatim, and the $U$-part is
Lemma~\ref{lem:m-flat}(a) with Lemma~\ref{lem:m-vanish}; hence $Q$
vanishes on the new kernel (Theorem~\ref{thm:vanish}, whose proof is
integrality-free), the child has $\calF=[r]$, and its (P)-leaf
solves a mixed integer linear program whose feasible set contains
$x^{*}$, returning the optimal value. If $\calG=\emptyset$ as well,
then $\calF=[r]$ at the current node and the same conclusion applies
directly.

Finally, every appended row is linearly independent of the current
rows: (S$_\R$)- and (S$_\Z$)-rows lie outside $\rowsp(C)$ by
construction, (CS)-rows by the choice of $J$, and batch rows by the
choice of $\calG'$. Each step that appends rows increases the rank of $C$. The only
possible step with no rank increase is the single transition to the
flat phase when the stationarity batch is empty. Thus there are at
most $n$ rank-increasing steps and at most one additional phase
transition. A path with $U=\{0\}$ is marked continuous-flat immediately, so
the initialization cannot repeat. This proves termination.
\end{proof}
\subsection{Running time}
It remains to prove the running-time bound.

\begin{proof}[Proof of Theorem~\ref{thm:m-main}]
By the proof of Proposition~\ref{prop:m-correct}, every root-to-leaf
path consists of at most $q$ (S$_\R$)-steps, at most one (CS)-step,
at most $n$ (S$_\Z$)-steps, at most one (Batch$_\Z$)-step, and a
leaf. Summing over the possible positions of the continuous
batch and the integer batch contributes at most $(n+1)^2$;
at each position use the product of the corresponding branching
factors below. This also bounds the total number of nodes up to
a factor $n+2$. We bound the number of children created at each step and the
growth of $\Delta:=\Delta(C)$ along the path; the number of leaves is
at most the product of the per-step child counts, and all linear algebra on any fixed node has polynomial bit
complexity in its data. At each continuous-flat node the family $\calH$ is
constructed by enumerating supports of size at most $\dim U^0+1$, at a
cost $(m+n)^{O(q+1)}\mathrm{poly}(\varphi)$; this initialization
factor will be included explicitly. It is incurred even if that
node subsequently terminates without creating children.

\emph{(S$_\R$).} Only rows of $A$ are appended, so throughout this
phase $C$ consists of rows of $A$ and $\Delta\le\Delta(A)$. Each step
creates at most $m$ children and enumerates no values, contributing
$m^{q}$ over the phase.

\emph{(CS).} One child. By Lemma~\ref{lem:m-augdelta} applied at this
node, where the entries of $C$ are bounded by $L_A$, the basis
vectors satisfy $\|u_j\|_\infty\le q^{q/2}L_A^{q}$, so the appended
rows $2u_j^TQ$ have entries at most
$2n\,q^{q/2}L_A^{q}L_Q=(nL)^{O(q+1)}$. All $|J|\le q$ rows are appended
at once, and a Laplace expansion of any square submatrix of $C^{0}$
along the appended rows expresses its determinant as a sum of at most
$n^{q}$ terms, each a product of at most $q$ appended entries and a
minor of the old $C$; hence
$\Delta(C^{0})\le\Delta_1:=\Delta(A)\cdot(nL)^{O((q+1)^{2})}$ and
$\delta_0\le\Delta(C^{0})\le\Delta_1$ by Lemma~\ref{lem:m-cramer}.

\emph{(S$_\Z$).} By Lemma~\ref{lem:m-proj}(ii), every appended row
$h$ is integral with $\|h\|_\infty\le(nL)^{O((q+1)^{2})}$, so by the
cofactor expansion in the proof of Lemma~\ref{lem:growth} each step
multiplies $\Delta$ by at most $n\|h\|_\infty=(nL)^{O((q+1)^{2})}$;
right-hand-side denominators never enter $\Delta$, which is computed
on integer data throughout (discussion after
Lemma~\ref{lem:m-cramer}). After $t\le n$ steps,
$\Delta\le\Delta_1(nL)^{O(t(q+1)^{2})}\le\Delta_2:=
\Delta_1(nL)^{O(n(q+1)^{2})}$. The children of the $t$-th step are
indexed by $h\in \calH$ and an enumerated value: the count of extreme-ray
supports gives $|\calH|\le m^{q+1}\le(m+n)^{q+1}$, and the interval
$[g_h-\|h\|_1\Delta,\,g_h]$ contains at most
$\delta_0\|h\|_1\Delta+1$ points of $\tfrac1{\delta_0}\Z$. The
product of the child counts over the phase is at most
\[
\prod_{t=1}^{n}(m+n)^{q+1}
\bigl(\delta_0\|h\|_1\Delta_1(nL)^{O(t(q+1)^{2})}+1\bigr)
\;\le\;
(m+n)^{n(q+1)}
\bigl(2\delta_0\,n(nL)^{O((q+1)^{2})}\Delta_1\bigr)^{n}
(nL)^{O(n^{2}(q+1)^{2})} .
\]

\emph{(Batch$_\Z$).} With $\Delta\le\Delta_2$, the adjugate basis
satisfies $\|y_i\|_\infty\le\Delta_2$, so
$y_i^TQy_i\le n^{2}L_Q\Delta_2^{2}$, and each of the $|\calG'|\le n$
rows enumerates at most $2\delta_0y_i^TQy_i+1\le
2\Delta_1n^{2}L_Q\Delta_2^{2}+1$ values of $\tfrac1{\delta_0}\Z$, for
a factor of at most $(3\Delta_1n^{2}L_Q\Delta_2^{2})^{n}$.

\emph{Leaf.} A type-1 leaf costs $\mathrm{poly}(\varphi)$. A
(P)-leaf is a mixed integer linear program with $p\le n$ integer
variables, solved by fixed-dimension integer programming methods
\cite{Lenstra83,Kannan87,Dadush12} in
$n^{O(n)}\cdot\mathrm{poly}(\varphi)$ time, using the residual
objective of Corollary~\ref{cor:linear}. These methods apply to the
projection onto the integer coordinates: linear programming supplies
the required separation oracle for each objective-threshold slice
and recovers the continuous coordinates. Exact optimization follows
from the standard rational bit bounds and threshold search.
The stationarity consistency checks use rational linear algebra and
add only polynomial work per node.
The rational right-hand sides and that objective may depend on
$b,c$, but their bit lengths are polynomial in $\varphi$ and in
the logarithms of the displayed coefficient bounds. For example,
post-batch entries are at most $2nL_Q\Delta_2$ and their minors
are bounded by $(2n^2L_Q\Delta_2)^n\Delta_2$. The logarithm of
this bound is polynomial in $n,q,\log(nL)$. Projected right-hand
sides are integer combinations of the original $b$, and all
later right-hand sides have the single denominator $\delta_0$;
their numerators have the same kind of polynomial bit bound.
Since $n,q\le\varphi$, these are uniform polynomial bounds in
$\varphi$, with no hidden dimension-dependent polynomial degree.

\emph{Collection.} Using $\Delta(A)\le(nL_A)^{n}$ and
$m\le(2L_A+1)^{n}$ after removing duplicate rows, the initialization of $\calH$ contributes at most
$(m+n)^{O(q+1)}=(nL)^{O(n(q+1))}$ per continuous-flat node, which is
absorbed in the final bound. All other factors above
are powers of $nL$: $m^{q}=(nL)^{O(nq)}$,
$(m+n)^{n(q+1)}=(nL)^{O(n^{2}(q+1)^{2})}$,
$\Delta_1=(nL)^{O(n+(q+1)^{2})}$, and
$\Delta_2=(nL)^{O(n+n(q+1)^{2})}$. The dominant contributions are
$(nL)^{O(n^{2}(q+1)^{2})}$, from the $n$-fold product of the growing
$\Delta$ in the (S$_\Z$)-phase and from
$\Delta_2^{2n}=(nL)^{O(n^{2}(q+1)^{2})}$ in the batch. Since
$q^{2}+q+1\le(q+1)^{2}$, every exponent is $O(n^{2}(q+1)^{2})$, including the
initialization factor and the $n+1$ possible phase-transition
positions,
giving the stated bound
$(nL)^{O(n^{2}(q+1)^{2})}\cdot\mathrm{poly}(\varphi)$.

At $q=0$ the algorithm reduces to the pure one and the coarse
bound becomes $(nL)^{O(n^2)}$:
$U=\{0\}$ at the root, so the path begins continuous-flat with $\delta_0=1$ and
$D$ empty; the cone $\Lambda$ is all of $\R^{m}_{\ge0}$, whose
extreme rays are the coordinate directions, so $\calH$ is the set of rows
of $A$ with $g_h=b_h$; (S$_\Z$) becomes the constraint branch,
(Batch$_\Z$) the batch, and (P) the ILP leaf of
Algorithm~\ref{alg:batch}, recovering Theorem~\ref{thm:main}.
\end{proof}

For fixed $q$, the bound has the same coarse form as in the pure case.
The $(q+1)^2$ factor comes from the size of the projected rows and
subsequent determinant growth; improving it is an open question
(Section~\ref{sec:discussion}).

\section{Unboundedness and certificates}\label{sec:unbounded}
The optimization algorithms branch on numerical slack and gradient
levels. Unboundedness admits a different treatment: count recession
faces and rays, and test signs by linear optimization. We obtain a
coefficient-magnitude-independent running time for every fixed dimension.
The certificate framework comes from Del Pia, Dey, and Molinaro
\cite{DelPiaDeyMolinaro17}; the face enumeration and refinement below
provide the explicit running-time bound. Lokshtanov's algorithm also
detects unboundedness, with a parametric factor depending on coefficient
magnitudes \cite{Lokshtanov17}.

A finite optimal value does not require a bounded feasible polyhedron.
The test here classifies arbitrary rational polyhedra; the mixed
batching theorem in Section~\ref{sec:mixed} assumes a polytope.

\begin{theorem}[Unboundedness test]\label{thm:unb-main}
There is an algorithm that, given an instance of IQP~\eqref{eq:iqp},
decides whether it is infeasible, unbounded, or has a finite optimal
value, in time
\[
(m+n)^{O(n)}\cdot\mathrm{poly}(\varphi).
\]
If unbounded, the algorithm outputs an integer point
$x_{0}\in P\cap\Z^{n}$ and an integer ray
$\rho\in\mathrm{rec}(P)\cap\Z^{n}$, both of size
$\mathrm{poly}(\varphi)$, such that
$f(x_{0}+t\rho)\to-\infty$ as $t\to\infty$ over the integers.
In particular the parametric factor is independent of
$L_{A},L_{Q},\Delta(A)$: for every fixed $n$ the test runs in
polynomial time.
\end{theorem}

There are two possible certificates. A recession direction with negative
quadratic curvature gives quadratic decrease from any feasible basepoint.
If the recession form is nonnegative, unboundedness is witnessed by a
zero-curvature direction with negative linear slope at a feasible
basepoint. Test~A handles the first case; a finite refinement makes the
second case algorithmic.

\begin{corollary}[Reduction to finite-value optimization]\label{cor:unb-confined}
For every fixed $n$, deciding boundedness of IQP is polynomial;
hence the open question of polynomial-time exact optimization for
fixed $n$ concerns instances with a finite optimal value. Here
``bounded'' refers to the objective value, not necessarily the
feasible polyhedron.
\end{corollary}

\begin{proof}
Boundedness is decided by the test of Theorem~\ref{thm:unb-main} in
$(m+n)^{O(n)}\cdot\mathrm{poly}(\varphi)$ time, which is polynomial
for fixed $n$ and involves no parametric dependence on $L$. Any exact
algorithm may therefore run the test first and handle unbounded
instances outright, so every occurrence of $L$ in its parametric
factor is incurred on bounded instances.
\end{proof}

We contrast this with variable dimension, where deciding
unboundedness is NP-complete already for objective rank $3$ with no
integer variables \cite[Thm.~2]{DelPia22}, and with the concave
case in variable dimension, where Del Pia~\cite[\S2.2]{DelPia18} gives an
explicit test using $2k+1$ mixed-integer linear programs.

\subsection{Feasibility and pointed pieces} First test integer feasibility. Partition the
feasible region into the at most $2^n$ closed orthant pieces
\[
 P_\sigma=P\cap\{x:\sigma_i x_i\ge0\ (i\in[n])\},
 \qquad \sigma\in\{-1,1\}^n.
\]
Discard pieces with no integer feasible point. Their recession cones
are pointed: a vector and its negative in the same orthant must
both be zero. The original problem is unbounded if and only if one
of these finitely many integer subproblems is unbounded. Each piece
adds $n$ unit rows and preserves integrality and polynomial encoding
length. Thus it suffices below to consider a nonempty pointed piece,
which we again denote by $P=\{x:Ax\le b\}$; its row count is at most
the original $m+n$. No quotient of the integer lattice is taken.

Let $K=\mathrm{rec}(P)$ and
$P_I=\mathrm{conv}(P\cap\Z^n)$. Meyer's theorem~\cite{Meyer74}
gives a rational polyhedron satisfying
\begin{equation}\label{eq:rec-equal}
 \mathrm{rec}(P_I)=\mathrm{rec}(P)=K.
\end{equation}
In particular $P_I$ is pointed. Write $q(r)=r^TQr$ and
\[
 s_w(x)=2w^TQx+c^Tw.
\]
Then $f(x+tw)=f(x)+t\,s_w(x)+t^2q(w)$.

\subsection{Negative-curvature test}
The first test isolates negative curvature on the recession cone.

\begin{lemma}[Negative-curvature test]\label{lem:unb-testA}
The value $\mu:=\min\{q(r):Ar\le0,\ r\in[-1,1]^{n}\}$ is rational of
size $\mathrm{poly}(\varphi)$, computable in
$(m+n)^{O(n)}\cdot\mathrm{poly}(\varphi)$ time together with a
rational minimizer; and $\mu<0$ iff some $r\in K$ has $q(r)<0$.
If $\mu<0$, IQP is unbounded, with certificate $(x_{0},\rho)$ where
$\rho$ is the integrally scaled minimizer and $x_{0}\in P\cap\Z^{n}$
is arbitrary.
\end{lemma}

\begin{proof}
The feasible region $K\cap[-1,1]^{n}$ is a rational polytope with
$m+2n$ facets, so $q$ attains its minimum, and by homogeneity of $q$
the sign claim holds. For the computation, enumerate the faces $F$
(at most $(m+2n+1)^{n}$, each specified by its tight rows $T$). If
the minimum is attained at $x^{*}$ with minimal face $F$, then $x^{*}$
lies in the relative interior of $F$ and is a critical point of
$q|_{\mathrm{aff}(F)}$. Parameterizing $\mathrm{aff}(F)$ as
$r_{0}+Vs$, criticality reads $V^{T}Q(r_{0}+Vs)=0$, a linear system;
on its solution set the objective is \emph{constant}: writing
$g(s)=s^{T}Ms+p^{T}s+g_{0}$ with $M=V^{T}QV$, two critical points
$s_{0}$ and $s_{0}+k$ satisfy $2M(s_{0}+k)+p=2Ms_{0}+p=0$, so
$Mk=0$ and
$g(s_{0}+k)=g(s_{0})+k^{T}(2Ms_{0}+p)+k^{T}Mk=g(s_{0})$.
So each face contributes one well-defined candidate value, provided
its critical set intersects the face; this intersection is a linear
feasibility problem, checked by one LP. The minimum over all
contributing faces equals $\mu$, all arithmetic is on rational data
of size $\mathrm{poly}(\varphi)$, and a rational minimizer is a basic
solution of the corresponding LP. If $\mu<0$, scale the minimizer by
the least common denominator to obtain $\rho\in K\cap\Z^{n}$ with
$q(\rho)<0$; for any $x_{0}\in P\cap\Z^{n}$ (one Lenstra call),
$x_{0}+t\rho\in P\cap\Z^{n}$ for all $t\in\Z_{\ge0}$ and
$f(x_{0}+t\rho)=t^{2}q(\rho)+O(t)\to-\infty$.
\end{proof}

From now on assume $\mu\ge0$, i.e., $Q$ is \emph{copositive on $K$}.
When the test passes, $q$ is nonnegative on $K$, and the difficulty
concentrates on the directions where it vanishes.

Call $r\in K$ \emph{flat} if $q(r)=0$.

\begin{lemma}[Flat cross-term lemma]\label{lem:unb-cross}
Let $Q$ be copositive on $K$, let $r\in K$ be flat, and let $w\in K$.
Then $r^{T}Qw\ge0$.
\end{lemma}

\begin{proof}
For $t>0$, $r+tw\in K$, so
$0\le q(r+tw)=2t\,r^{T}Qw+t^{2}q(w)$; divide by $t$ and let
$t\downarrow0$.
\end{proof}

\subsection{A finite family of flat directions}
The cross-term lemma controls interactions with zero-curvature rays.
We next construct a finite collection of such rays sufficient to
witness every negative-slope case.

\begin{lemma}[Decomposition by chains of faces]\label{lem:unb-triang}
For a recession cone $K$ contained in an orthant, one can compute a
simplicial decomposition with at most $(m+n+1)^{O(n)}$ pieces in
$(m+n+1)^{O(n)}\mathrm{poly}(\varphi)$ time. All generators have
polynomial encoding length, with a polynomial degree independent of
$n$, and can be scaled to integer vectors.
\end{lemma}
\begin{proof}
If $K=\{0\}$, return the single zero-dimensional cone with no
generators. Otherwise fix its
orthant sign vector $\sigma$ and put $a=\sigma$. On $K$,
$a^Tr=\|r\|_1>0$ for $r\ne0$. The section
$B=K\cap\{a^Tr=1\}$ is a rational polytope of dimension
$d-1$, where $d=\dim K$, described by at most $m$ inequalities
and its affine equations.

Enumerate its nonempty faces by subsets of at most $n$ defining
inequalities, using linear programming to discard empty intersections
and identify the rows identically tight on each face. This is
$(m+n+1)^{O(n)}$ work, including removal of duplicates. For each
face $F$ choose a rational point $p_F$ in its relative interior.
Such a point is obtained by maximizing a common slack $\eta$, with
$0\le\eta\le1$, in all rows not identically tight on $F$, while
imposing its affine equations. The optimum is positive whenever
such rows exist; a rational optimal basic solution has polynomial
encoding length. If no such rows exist, any rational feasible
point in the affine set suffices.

For each maximal chain of nonempty faces of $B$, take the convex
hull of their chosen points. These are simplices that cover $B$:
inductively subdivide each proper face in the same manner and cone
its boundary subdivision to $p_B$. The points in a chain are
affinely independent because each new point is outside the affine
hull of the preceding proper face. A maximal descending chain can
be specified by successively selecting one original inequality that
cuts the current face to a facet. There are $d-1$ selections, so
there are at most $\max(1,m)^{d-1}$ chains (choose a canonical
selection when several rows define the same facet). Coning these
simplices from the origin gives the required simplicial cones.

Every point $p_F$ is obtained directly from the original section
data, rather than from earlier chosen points. Standard determinant
bounds for rational linear systems therefore give a uniform
polynomial bound in $\varphi$ for every generator. Clearing the
denominators of at most $n$ coordinates preserves such a bound.
\end{proof}

For a simplicial cone $S=\mathrm{cone}(w_{1},\dots,w_{k})$, the faces
of $S$ are exactly the subcones $\mathrm{cone}(w_{i}:i\in I)$,
$I\subseteq[k]$; this is the property the refinement exploits.

\begin{lemma}[Refinement with a dimension-decreasing recursion]
\label{lem:unb-refine}
Assume $Q$ is copositive on $K$. Let $S_1,\ldots,S_N$ be the
decomposition from Lemma~\ref{lem:unb-triang}. In
$(m+n+1)^{O(n)}\mathrm{poly}(\varphi)$ time one can compute a
refinement into at most $n!N$ simplicial cones with rational
generators of polynomial encoding length such that every nonzero
face $F$ of every resulting cone satisfies
\[
 \min\{q(r):r\in F,\ a^Tr=1\}=0
 \quad\Longrightarrow\quad
 F\text{ has a flat extreme ray}.
\]
Here $a$ is the fixed orthant sign vector, so $a^Tr=\|r\|_1$ on
every cone under consideration.
\end{lemma}
\begin{proof}
For a zero-dimensional piece return that piece without generators.
For a positive-dimensional piece we use the face induction underlying
\cite[Lemma~7]{DelPiaDeyMolinaro17}, with an explicit computation
and count. Fix an original simplicial piece $S$ of dimension $d$.
The recursive inputs are faces $T$ of this original $S$. Minimize
$q$ on the simplex $T\cap\{a^Tr=1\}$ using the stationary-set
enumeration of Lemma~\ref{lem:unb-testA}. The same proof applies
to this section: it has at most $2^d$ faces. If the minimum is
positive, return $T$, since all its nonzero faces have positive
minimum. If the minimum is zero, choose a rational minimizer $z_T$.
For every facet $G$ of $T$ not containing $z_T$, recursively refine
$G$ and return the cones $\mathrm{cone}(z_T,R)$, where $R$ ranges
over the resulting pieces of $G$.

This recursion decreases the dimension of its input at each call.
In dimension one, return the ray. Coverage follows by extending
the segment from $z_T$ through a point of the section to a facet
not containing $z_T$, with boundary cases obtained by closure.
Each returned cone is simplicial because $z_T$ lies outside the
linear span of $G$. A face of $\mathrm{cone}(z_T,R)$ either
contains the flat extreme ray $z_T$, or is a face of $R$ and
inherits the required property by induction.

There are at most $d$ facets at dimension $d$. The number of
output cones thus obeys $b(1)=1$ and $b(d)\le d\,b(d-1)$, giving
$b(d)\le d!$. The number of recursive calls is $d^{O(d)}$, and
each call requires at most $2^d$ linear feasibility tests on
stationary sets. Crucially, all $T$ are faces of the original
$S$: each new ray is either an original generator or a minimizer
computed directly on such a face. There is no repeated
substitution of newly computed rays into the minimization data.
The input data of $S$, its face equations, the stationary systems,
and their rational basic solutions all have uniformly polynomial
encoding length. Multiplying this work by the initial count $N$
proves the asserted time and size bounds.
\end{proof}

\begin{corollary}[Quantitative positivity off the flat rays]
\label{cor:unb-eps}
Let $S'=\mathrm{cone}(w_{1},\dots,w_{k})$ be a refined piece,
$J=\{i:q(w_{i})=0\}$, and $S'_{\perp}:=\mathrm{cone}(w_{i}:i\notin J)$.
Then there is a rational $\varepsilon>0$ of size
$\mathrm{poly}(\varphi)$ with
$q(r)\ge\varepsilon\|r\|^{2}$ for all $r\in S'_{\perp}$.
\end{corollary}

\begin{proof}
If $S'_{\perp}=\{0\}$, take $\varepsilon=1$. Otherwise $S'_{\perp}$ is a face of $S'$ whose extreme rays are exactly
$\{w_{i}:i\notin J\}$, none flat. If
$\min\{q:\ r\in S'_{\perp},\ \|r\|_{1}=1\}$ were $0$,
Lemma~\ref{lem:unb-refine} would give a flat extreme ray of
$S'_{\perp}$, a contradiction; so the minimum is positive, rational
of polynomial size by the computation of Lemma~\ref{lem:unb-testA},
and homogeneity gives the claim (adjusting norms costs a factor
$n$).
\end{proof}
\subsection{Characterization and algorithm}
With the refined decomposition in place, unboundedness admits an
exact characterization.

\begin{theorem}[Characterization of unboundedness]\label{thm:unb-char}
Let $P$ be a pointed orthant piece with $P\cap\Z^{n}\neq\emptyset$. IQP~\eqref{eq:iqp} is unbounded if
and only if
\begin{itemize}
\item[(A)] some $r\in K$ has $q(r)<0$; or
\item[(B)] $Q$ is copositive on $K$, and some \emph{flat integer
extreme ray} $w$ of the refined decomposition of
Lemma~\ref{lem:unb-refine} admits a point $x\in P\cap\Z^{n}$ with
$s_{w}(x)\le-1$.
\end{itemize}
\end{theorem}

\begin{proof}
\emph{Sufficiency.} (A) is Lemma~\ref{lem:unb-testA}. For (B):
$x+tw\in P\cap\Z^{n}$ for $t\in\Z_{\ge0}$ and, since $w$ is flat,
$f(x+tw)=f(x)+t\,s_{w}(x)\to-\infty$. (With integer data,
$s_{w}(x)=2w^{T}Qx+c^{T}w\in\Z$ at integer $x$ and integrally scaled
$w$, so $s_{w}(x)<0$ iff $s_{w}(x)\le-1$.)

\emph{Necessity.} Suppose (A) fails, so $Q$ is copositive on $K$, and
let $x_{k}\in P\cap\Z^{n}$ with $f(x_{k})\to-\infty$; then
$\|x_{k}\|\to\infty$. The key step is to decompose over the
\emph{integer hull}: by Meyer's theorem and \eqref{eq:rec-equal},
\[
x_{k}\;=\;v_{k}+r_{k},
\qquad
v_{k}\in\mathrm{conv}\bigl(\mathrm{vert}(P_{I})\bigr),
\quad r_{k}\in K .
\]
The $v_{k}$ range over a fixed bounded set (vertices of $P_{I}$
have finite norm), so all quantities of the form $v_{k}^{T}Qw$,
$f(v_{k})$ below are bounded by a constant $\Gamma$ depending on the
instance only. Passing to a subsequence, all $r_{k}$ lie in one
refined piece $S'=\mathrm{cone}(w_{1},\dots,w_{k'})$; write
\[
r_{k}\;=\;r_{k}^{J}+r_{k}^{\perp},
\qquad
r_{k}^{J}=\sum_{i\in J}\lambda_{i,k}w_{i}\in\text{(flat rays)},
\quad
r_{k}^{\perp}=\sum_{i\notin J}\lambda_{i,k}w_{i}\in S'_{\perp},
\quad \lambda\ge0 .
\]
Expand $f(x_{k})=f(v_{k})+q(r_{k})+\ell_{k}(r_{k})$ where
$\ell_{k}(r):=2v_{k}^{T}Qr+c^{T}r=s_{r}(v_{k})$ (linear in $r$).
Now bound each piece from below:
\begin{itemize}
\item $q(r_{k}^{J})\ge0$: it is a nonnegative combination of
flat--flat cross terms, each $\ge0$ by
Lemma~\ref{lem:unb-cross}.
\item $2(r_{k}^{J})^{T}Qr_{k}^{\perp}\ge0$: flat directions have
nonnegative cross terms with all of $K\supseteq S'_{\perp}$
(Lemma~\ref{lem:unb-cross} again, applied to each $w_{i}$, $i\in J$).
\item $q(r_{k}^{\perp})\ge\varepsilon\|r_{k}^{\perp}\|^{2}$
(Corollary~\ref{cor:unb-eps}).
\item $\ell_{k}(r_{k}^{\perp})\ge-\Gamma'\|r_{k}^{\perp}\|$ with
$\Gamma'$ instance-bounded.
\item $f(v_{k})\ge-\Gamma$.
\end{itemize}
Hence
\[
f(x_{k})
\;\ge\;
\ell_{k}(r_{k}^{J})
+\varepsilon\|r_{k}^{\perp}\|^{2}-\Gamma'\|r_{k}^{\perp}\|-\Gamma
\;\ge\;
\ell_{k}(r_{k}^{J})-\Gamma'' ,
\]
so $f(x_{k})\to-\infty$ forces
$\ell_{k}(r_{k}^{J})=\sum_{i\in J}\lambda_{i,k}\,s_{w_{i}}(v_{k})
\to-\infty$, and since each $s_{w_{i}}(v_{k})$ is bounded, some
$i^{*}\in J$ has $s_{w_{i^{*}}}(v_{k})<0$ for infinitely many $k$.

It remains to convert the \emph{continuous} basepoint $v_{k}$ into an
\emph{integer feasible} one, and this is exactly what decomposing
over $P_{I}$ buys: $s_{w_{i^{*}}}$ is linear and
$v_{k}\in\mathrm{conv}(\mathrm{vert}(P_{I}))$, so
$s_{w_{i^{*}}}(v_{k})\ge\min\{s_{w_{i^{*}}}(u):u\in\mathrm{vert}(P_{I})\}$,
so some vertex $u$ of $P_{I}$, an \emph{integer feasible
point}, satisfies $s_{w_{i^{*}}}(u)<0$, i.e.\ $\le-1$ after
integral scaling of $w_{i^{*}}$. This is condition (B).

(Had we decomposed over $P$ itself, the basepoints would be vertices
of $P$, generally non-integral, and the negative slope could fail at
every integer point; the integer-hull decomposition is what closes
the gap. Note also that boundedness of $s_{w}$ from below on $P$ is
automatic: a recession direction $u$ with $2w^{T}Qu<0$ would
contradict Lemma~\ref{lem:unb-cross}.)
\end{proof}
Assembling the tests proves the main theorem of the section.

\begin{proof}[Proof of Theorem~\ref{thm:unb-main}]
Test integer feasibility and form the orthant pieces described
above. On every integer-feasible piece perform the following steps.
First run Test~A; a negative value supplies the quadratic-decrease
certificate. Otherwise apply Lemmas~\ref{lem:unb-triang} and
\ref{lem:unb-refine} and scale each flat generator $w$ to an integer
vector. For every such generator solve
\[
 \exists x\in P\cap\Z^n:\quad
 (2Qw)^Tx\le-c^Tw-1.
\]
A feasible solution gives a linear-decrease certificate $(x,w)$.
If no piece supplies a certificate, report a finite optimal value;
Theorem~\ref{thm:unb-char} proves boundedness below in every piece.
An integer-valued objective on a nonempty integer feasible set
attains its infimum whenever it is bounded below.

There are at most $2^n$ pieces, each with at most $m+n$ original
and orthant rows, and at most $(m+n)^{O(n)}$ refined generators.
Each integer linear feasibility call takes
$n^{O(n)}\mathrm{poly}(\varphi)$ time. The preceding lemmas give
uniform polynomial encoding bounds for the ray data and hence for
these linear systems and their small feasible witnesses. The
overall time is $(m+n)^{O(n)}\mathrm{poly}(\varphi)$, and the
returned basepoint and integer direction have polynomial encoding
length. All certificates lie in the original feasible set and its
recession cone because each orthant piece is a subset of that set.
\end{proof}

\subsection{Mixed-integer unboundedness}
The recession geometry is unchanged in the mixed case. The basepoint
belongs to $\Z^p\times\R^q$, however, so a negative slope need not be
an integer. We therefore minimize each slope by MILP and compare its
value with zero exactly.

\begin{corollary}[Mixed-integer unboundedness]\label{cor:m-unb}
For a rational polyhedron $P=\{x:Ax\le b\}$ with integer
quadratic data, one can decide whether MIQP~\eqref{eq:miqp0} is
infeasible, unbounded below, or has a finite attained optimum in
$(m+n)^{O(n)}\mathrm{poly}(\varphi)$ time. In the unbounded case
one can output a rational point
$x_0\in P\cap(\Z^p\times\R^q)$ and an integer nonzero direction
$\rho\in\mathrm{rec}(P)$, both of polynomial encoding length,
such that $f(x_0+t\rho)\to-\infty$ for integer $t\to\infty$.
Under ETH no $(m+\varphi)^{o(n)}\mathrm{poly}(\varphi)$ algorithm
exists, and unless $\mathrm{FPT}=\mathrm{W[1]}$ no
$f(n)\mathrm{poly}(m,\varphi)$ algorithm exists.
\end{corollary}
\begin{proof}
Test mixed-integer feasibility and split into orthant pieces as in
Section~\ref{sec:unbounded}. For each nonempty piece, Meyer's
theorem~\cite{Meyer74} gives a rational mixed-integer hull
\[
 P_I^{\mathrm{mix}}=
 \mathrm{conv}(P\cap(\Z^p\times\R^q)),\qquad
 \mathrm{rec}(P_I^{\mathrm{mix}})=\mathrm{rec}(P).
\]
It is pointed, and each vertex belongs to the mixed feasible set:
an extreme point of its convex hull can only be represented as a
finite convex combination of feasible points if every point with
positive weight equals that extreme point. Test~A and the
decomposition and refinement depend only on the rational recession
data and remain unchanged. In the necessity proof of
Theorem~\ref{thm:unb-char}, use $P_I^{\mathrm{mix}}$ to obtain a
mixed feasible vertex with $s_w(x)<0$ for some refined flat ray.

For every flat integer generator $w$, solve the MILP
\[
 \eta_w=\min\{s_w(x):x\in P\cap(\Z^p\times\R^q)\}.
\]
It is bounded below: for every $r\in\mathrm{rec}(P)$,
the cross-term lemma gives $2w^TQr\ge0$, so the linear functional
is bounded below even on $P$. A bounded-below rational MILP on a
nonempty rational mixed set attains a rational optimum and is
solvable in $n^{O(n)}$ times a fixed polynomial in its input
length. Test $\eta_w<0$ exactly. If so, an optimal mixed point
and $w$ form a certificate. If all tests fail and Test~A finds no
negative curvature, the same expansion used in
Theorem~\ref{thm:unb-char} proves boundedness below.

For completeness, finite MIQP optima on rational polyhedra are
attained. Fix an integer fiber $x_I=z$. If the full problem is
bounded below, the fiber QP is bounded below; the Frank--Wolfe
attainment theorem for quadratic functions on polyhedra applies~\cite{FrankWolfe56}.
Choose an optimal point and the minimal face of its fiber.
Stationarity on that face, together with its active equations,
gives a rational linear system with coefficients determined by
$(A,Q)$ and an integer right-hand side determined by $(b,c,z)$
after a fixed denominator clearing. The objective is constant on
the stationary set of this face. Intersect that set with the
fiber polyhedron and choose a rational basic feasible point after
orthant splitting if needed. There are finitely many possible
coefficient matrices, all determined by $(A,Q)$ and the fixed
orthant rows. This is an existence argument over the finitely
many choices of active rows, not an enumeration of integer
fibers. Let $D_0$ be a common multiple of their nonzero
square determinants and of the fixed denominator clearings.
Every nonempty fiber therefore has an optimal point in
$\tfrac1{D_0}\Z^n$ and an optimal value in
$\tfrac1{D_0^2}\Z$. A nonempty bounded-below subset of this
discrete set has a least element. The denominator is used only
for this existence argument and is not enumerated by the test.

The number of pieces and rays and the encoding length of every
MILP are as in Theorem~\ref{thm:unb-main}, giving the same time
bound. The lower bounds follow from the pure-integer instances in
Theorem~\ref{thm:lb-unbounded}.
\end{proof}

\begin{remark}[Strict slopes in mixed problems]
Integer input data do not make $s_w(x)$ integer when $x$ has
continuous coordinates. For example, $\min -2zy$ subject to
$z\in\Z_{\ge0}$ and $3y=1$ is unbounded, but for the flat ray
$w=(1,0)$ its slope is always $-2/3$. Testing $s_w(x)\le-1$
would miss this instance. The exact MILP optimum test above
avoids any choice of a universal strictness threshold.
\end{remark}

\section{Dimension-dependent lower bounds}\label{sec:lb}
We match the linear order of dimension in the exponent of the
recession test. A single chord construction first gives a threshold
problem over a polytope, with integral YES witnesses and a uniform
gap. It then transfers to bounded IQP, and, after homogenization,
to noncopositivity and integer unboundedness. The continuous hardness
is known from Knauer--K\"onig--Werner~\cite{KKW15}, and IQP hardness
in the dimension is established by Herrmann~\cite{Herrmann26}.
Here the arbitrary-set gadget, its continuous gap, and controlled-inertia
transfers provide the specific lower bounds needed for our recession test.

\subsection{Source problem and encoding}
Our source is \textup{\textsc{Multicolored $k$-Sum}}: given $k$ finite sets
of positive integers $N_1,\ldots,N_k$, each of size at most $m_0$,
and an integer target $t$, select one number from each set summing
to $t$. 
This problem is W[1]-hard parameterized by $k$
\cite{AbboudLewiWilliams14}. One can obtain the needed many-one
colored form directly from multicolored clique. Use one selection
slot for each of its $h$ vertex colors and one for each of its
$\binom h2$ edge-color pairs. Give each vertex in a color a
distinct positive identifier at most $M$. For each endpoint
incidence between a vertex color and an edge-color pair, reserve
one digit. A vertex selection contributes its identifier in all
its incidence digits; an edge selection contributes $M$ minus
its endpoint identifier in the corresponding digit. The target
has value $M$ in every incidence digit. In a base larger than
$2M$, there is no carry, so equality of sums requires each
chosen edge to have exactly the chosen endpoints. All
unspecified digits are zero. Adding $1$ to every encoded number
and adding the number of slots to the target makes every
number positive without changing equality. The resulting number
of slots is $k=h+\binom h2$, a function of $h$, and the
construction has polynomial encoding length. Thus it proves
many-one W[1]-hardness, without relying on an OR of instances.

For the ETH statement use the bounded-bit source in
\cite[Theorem~5.1 and Corollary~5.1]{PatrascuWilliams10}:
an $N^{o(k)}$ algorithm for $k$-SUM on numbers of
$O(k\log N)$ bits would refute ETH. Their SAT-to-sum construction
already selects one partial assignment from each of $k$ blocks,
so these blocks give the colors without increasing $k$.
The integers and target can be made positive by a common shift
and the corresponding target shift, preserving the bit bound.
We use this restricted source for every ETH claim below.

\subsection{Chord construction and margin}
Fix a color class $N=\{\nu_{1}<\dots<\nu_{s}\}\subseteq\Z$. For each
consecutive pair define the \emph{chord line}
\[
\ell_{i}(u)\;:=\;(\nu_{i}+\nu_{i+1})\,u\;-\;\nu_{i}\nu_{i+1}\;-\;1,
\qquad i=1,\dots,s-1,
\]
i.e., the affine function through the points
$(\nu_{i},\nu_{i}^{2}-1)$ and $(\nu_{i+1},\nu_{i+1}^{2}-1)$ on the
downshifted parabola. Let
$\mathrm{env}(u):=\max_{i}\ell_{i}(u)$ on $[\nu_{1},\nu_{s}]$ and
\[
\phi(u)\;:=\;\mathrm{env}(u)-u^{2}.
\]

\begin{figure}[t]
\centering
\begin{tikzpicture}[scale=0.85]
  \draw[->] (0.1,-0.6) -- (5.4,-0.6) node[right] {$u$};
  \draw[dashed,domain=0.55:4.85,smooth,variable=\u]
    plot ({\u},{(\u*\u-1)/3.2}) node[right,yshift=-5pt,font=\small] {$u^{2}-1$};
  \draw[thin,domain=0.35:4.75,smooth,variable=\u]
    plot ({\u},{\u*\u/3.2}) node[above left,yshift=3pt,font=\small] {$u^{2}$};
  \draw[very thick]
    (1,0) -- (2,{3/3.2}) -- (3.5,{11.25/3.2}) -- (4.5,{19.25/3.2});
  \foreach \a in {1,2,3.5,4.5}
    \fill ({\a},{(\a*\a-1)/3.2}) circle (1.7pt);
  \foreach \a/\l in {1/{\nu_{1}},2/{\nu_{2}},3.5/{\nu_{3}},4.5/{\nu_{4}}}
    { \draw ({\a},-0.55) -- ({\a},-0.65);
      \node[below,font=\small] at ({\a},-0.62) {$\l$}; }
\end{tikzpicture}
\caption{The chord gadget for $N=\{\nu_{1},\dots,\nu_{4}\}$. The lower
boundary of the feasible region (thick) is the upper envelope of the
chords of $u^{2}-1$ (dashed); it touches $u^{2}-1$ exactly at the
elements of $N$, so $\phi(u)=\mathrm{env}(u)-u^{2}$ equals $-1$
exactly on $N$, while $\phi(u)+1$ has a lower bound proportional to the distance to $N$
(Lemma~\ref{lem:lb-margin}).}
\label{fig:gadget}
\end{figure}

\begin{lemma}[Uniform-depth membership penalty]\label{lem:lb-gadget}
On $[\nu_{1},\nu_{s}]$: $\phi(u)\ge-1$, with equality if and only if
$u\in N$. For $u\in[\nu_{i},\nu_{i+1}]$,
\[
\ell_{i}(u)-u^{2}\;=\;(u-\nu_{i})(\nu_{i+1}-u)\;-\;1 ,
\]
and $\ell_{i}$ attains the maximum in the definition of
$\mathrm{env}$ on that interval.
\end{lemma}

\begin{proof}
The displayed identity is a two-line expansion. On
$[\nu_{i},\nu_{i+1}]$ the product $(u-\nu_{i})(\nu_{i+1}-u)$ is
$\ge0$, vanishing exactly at the endpoints, so
$\ell_{i}(u)-u^{2}\ge-1$ with equality iff $u\in\{\nu_{i},\nu_{i+1}\}$.
For $j<i$, both endpoints $\nu_i,\nu_{i+1}$ lie on or to the
right of $\nu_{j+1}$; the same displayed identity shows
$\ell_j(\nu_i)\le\nu_i^2-1$ and
$\ell_j(\nu_{i+1})\le\nu_{i+1}^2-1$. For $j>i$ the analogous
inequalities hold because both endpoints lie on or to the left
of $\nu_j$. The affine function $\ell_i-\ell_j$ is therefore
nonnegative at both endpoints and throughout the interval.
Hence $\ell_i$ attains the envelope there, proving both claims. (Endpoints of
the box are elements of $N$, so box faces introduce no deeper
values.)
\end{proof}

The crucial features, unavailable to tangent-based gadgets: the depth
$-1$ is \emph{uniform} regardless of the spacing of $N$, and the
penalty is realized by linear inequalities ($v\ge\ell_i(u)$ for all $i$)
together with the \emph{concave} objective $v-u^{2}$: minimizing
pushes $v$ down onto $\mathrm{env}(u)$, so
$\min_{v}\,(v-u^{2})=\phi(u)$, and the pressure points the right way.
The gadgets compose into the threshold reduction.

An empty color class is an immediate NO-instance and may be
replaced by the fixed NO source $N_1=\{2,3\}$, $t=1$. Otherwise, given a
\textup{\textsc{Multicolored $k$-Sum}} instance, assume each color
class has $s_j\ge2$ elements; otherwise pad with a distinct positive dummy number larger than
both $t$ and the existing element, which can participate in no solution since all
numbers are positive, and which prevents the degenerate case of a
chordless slot whose $v_{j}$ would be unbounded below. Build
variables
$(u_{j},v_{j})_{j=1}^{k}$, $n=2k$, constraints
\[
P:\quad
\ell^{(j)}_i(u_j)\le v_j\le B_j\ \ \forall i,\qquad
\nu^{(j)}_{1}\le u_{j}\le\nu^{(j)}_{s_{j}},\qquad
\sum_{j=1}^{k}u_{j}=t ,
\]
where $B_j=\max_{\nu\in N_j}(\nu^2-1)$. The upper envelope in
each interval is a linear interpolation of its endpoint heights
and hence is at most $B_j$. Thus the new upper bounds preserve
the envelope minimizers, while making $P$ a polytope.
There are $m=O(km_0)$ defining inequalities, and we use the concave quadratic objective
\[
\psi(u,v)\;:=\;\sum_{j=1}^{k}\bigl(v_{j}-u_{j}^{2}\bigr),
\qquad \nu_{-}(\text{Hessian})=k .
\]

\begin{lemma}[Value characterization and margin]\label{lem:lb-margin}
$\displaystyle\min_{P}\psi\;=\;\min\Bigl\{\sum_{j}\phi_{j}(u_{j})
:\ \textstyle\sum_{j}u_{j}=t,\ u_{j}\in[\nu^{(j)}_{1},\nu^{(j)}_{s_{j}}]
\Bigr\}\;\ge\;-k$,
with equality iff the \textup{\textsc{$k$-Sum}} instance is a YES-instance.
On NO-instances $\min_{P}\psi\ge-k+\tfrac12$.
\end{lemma}

\begin{proof}
The first equality is the per-slot elimination of $v_{j}$ noted
above; the bound and the equality condition are
Lemma~\ref{lem:lb-gadget} summed over slots: $\sum\phi_{j}=-k$ forces
every $u_{j}\in N_{j}$, and then $\sum u_{j}=t$ is a YES-witness;
conversely a witness gives value $-k$. (No vertex analysis of $P$ is
required.) For the margin a pointwise argument suffices; no attainment or
exchange step is needed. Fix any feasible $u$. In slot $j$ let
$[a_{j},b_{j}]$ be the cell of $N_{j}$ containing $u_{j}$, let
$e_{j}\in\{a_{j},b_{j}\}$ be the nearer endpoint, and let
$d_{j}=|u_{j}-e_{j}|\ge0$. Writing $x=u_{j}-a_{j}$ and
$y=b_{j}-u_{j}$, we have $x+y=b_{j}-a_{j}\ge1$ (distinct integers)
and $d_{j}=\min(x,y)\le(x+y)/2$, so by the identity of
Lemma~\ref{lem:lb-gadget},
$\phi_{j}(u_{j})\ge d_{j}\,(x+y-d_{j})-1\ge d_{j}/2-1$. Summing,
$\sum_{j}\phi_{j}(u_{j})\ge-k+\tfrac12\sum_{j}d_{j}$. The $e_{j}$
form an endpoint selection with $\sum_{j}e_{j}\in\Z$ and
$\bigl|t-\sum_{j}e_{j}\bigr|=\bigl|\sum_{j}(u_{j}-e_{j})\bigr|
\le\sum_{j}d_{j}$. On a NO-instance $\sum_{j}e_{j}\ne t$, and both
are integers, so $\sum_{j}d_{j}\ge1$ and
$\min\psi\ge-k+\tfrac12$.
\end{proof}

\subsection{Continuous and bounded integer threshold problems}
\begin{theorem}[Continuous threshold hardness]\label{thm:lb-qp}
Deciding, for a concave quadratic $\psi$ with integer coefficients and a polytope given by linear inequalities
$P\subseteq\R^{n}$, whether $\min_{P}\psi\le\theta$, is W[1]-hard
parameterized by $n$; the reduction uses $n=2k$,
$m=O(k\,m_0)$, and $\theta=-k+\tfrac14$ (scale by $4$ for
integrality). Under ETH there is no
$(m+\varphi)^{o(n)}$-time algorithm.
\end{theorem}
\begin{proof}
The gadget and margin lemmas separate YES and NO instances at
$\theta=-k+1/4$. Multiplying the objective and threshold by $4$
makes both integral. The construction takes polynomial time and
uses $n=2k$ variables. If source integers and the target have at most $B$ bits,
every coefficient has $O(B+\log k)$ bits, since the construction
uses only sums and products of two source integers and the target.
For the ETH source above $B=O(k\log N)$ and $m_0\le N$.
Thus $m+\varphi\le k^{O(1)}N^{O(1)}$, and, in the source regime
$k<N^{0.99}$, an exponent $o(n)=o(k)$ gives $N^{o(k)}$ time.
This contradicts the cited source theorem. The continuous
W[1]-hardness and ETH exclusion were already established for
norm maximization in \cite[Theorem~1.2]{KKW15}; the construction
here supplies the particular integer witnesses and explicit gap
used in the transfers below.
\end{proof}

\begin{proposition}[Transfer to bounded integer instances]
\label{prop:lb-integer}
The threshold problem for \emph{bounded IQP}, given integer
$Q,c,A,b$ with $\{Ax\le b\}$ a polytope and $\theta\in\Z$, decide
$\min\{x^{T}Qx+c^{T}x:\ Ax\le b,\ x\in\Z^{n}\}\le\theta$, is
W[1]-hard parameterized by $n$, with no $(m+\varphi)^{o(n)}$
algorithm under ETH. The hard instances have inertia $(0,k,k)$ with
$n=2k$.
\end{proposition}

\begin{proof}
Restrict the instances of Theorem~\ref{thm:lb-qp} to
$(u,v)\in\Z^{2k}$ (the polytope is unchanged; $m$ and $\varphi$ are
unchanged up to constants). On YES-instances the continuous
optimum is attained at an \emph{integral} point: the witness has
$u_{j}=\nu_{j}\in\Z$ and $v_{j}=\ell_{i}(\nu_{j})=\nu_{j}^{2}-1\in\Z$,
so the integer minimum equals $-k$ (after clearing the scaling).
On NO-instances the integer minimum is at least the continuous
minimum, hence at least $-k+\tfrac12$ by
Lemma~\ref{lem:lb-margin}. Using the integer-coefficient objective $4\psi$ and integer
threshold $-4k+1$ separates the two cases: YES gives $-4k$ and
NO gives at least $-4k+2$. The added bounds $v_j\le B_j$ ensure
that the feasible region is a polytope.
\end{proof}

\subsection{Homogenization and unboundedness}
\begin{lemma}[Homogenization transfer]\label{lem:lb-homog}
With $P,\psi,\theta$ as above, introduce $s$ and set
\[
K=\left\{(u,v,s):
\begin{array}{ll}
v_j\ge(\nu_i^{(j)}+\nu_{i+1}^{(j)})u_j
       -(\nu_i^{(j)}\nu_{i+1}^{(j)}+1)s & \text{for all }i,j,\\
\nu_1^{(j)}s\le u_j\le\nu_{s_j}^{(j)}s,\quad v_j\le B_js
 & \text{for all }j,\\
\sum_j u_j=ts,\quad s\ge0 &
\end{array}\right\}.
\]
a rational polyhedral cone with $O(km_0)$ defining inequalities
in dimension $n+1$, and the homogeneous quadratic
$F(u,v,s):=-\sum_{j}u_{j}^{2}+s\sum_{j}v_{j}-\theta s^{2}$.
Then: $F$ takes a negative value on $K$ iff $\min_{P}\psi<\theta$.
\end{lemma}

\begin{proof}
For $s>0$, $F(u,v,s)=s^{2}\bigl(\psi(u/s,v/s)-\theta\bigr)$ with
$(u/s,v/s)\in P$, giving both directions for $s>0$. For
$s=0$ the scaled box forces $u=0$ (the two box inequalities
read $0\le u_{j}\le0$), the chord rows and upper bounds force $v_j=0$, and
$F(0,0,0)=0$; so the $s=0$ part contributes no negative
values and no false positives.
\end{proof}

\begin{theorem}[{Noncopositivity / Test A is W[1]-hard in the
dimension}]\label{thm:lb-copos}
Given integer $Q\in\Z^{n\times n}$ and $A\in\Z^{m\times n}$, deciding
$\exists r:\,Ar\le0,\ r^{T}Qr<0$ is W[1]-hard parameterized by $n$,
and admits no $(m+\varphi)^{o(n)}$ algorithm under ETH. The hard
instances in dimension $n=2k+1$ have inertia $(1,k+1,k-1)$.
Consequently the complementary copositivity decision problem is
co-W[1]-hard under parameterized many-one reductions.
\end{theorem}

\begin{proof}
Lemma~\ref{lem:lb-homog} composed with Theorem~\ref{thm:lb-qp}
(strictness is handled by Lemma~\ref{lem:lb-margin}: YES gives
$\min\psi=-k<\theta$ and NO gives $\min\psi\ge-k+\tfrac12>\theta$, both
strict).
Multiply $F$ by $4$ to obtain a symmetric integer matrix:
the $u_j^2$ coefficients become $-4$, the $s^2$ coefficient
becomes $4k-1$, and each $sv_j$ coefficient is $4$, represented
by the two equal matrix entries $Q_{s,v_j}=Q_{v_j,s}=2$. Positive scaling preserves signs
and inertia. The $u$-block contributes $k$ negative eigenvalues.
In the $(v,s)$-block, the subspace $s=0$, $\sum_jv_j=0$ has
dimension $k-1$ and is in the kernel. On the complementary
two-dimensional space the matrix has negative determinant,
because the coefficient of $s\sum_jv_j$ is nonzero, so it has
one positive and one negative eigenvalue. This gives
$(1,k+1,k-1)$. Taking complements of the many-one reduction gives
the stated co-W[1]-hardness of copositivity.
\end{proof}

\begin{theorem}[{IQP unboundedness is W[1]-hard in the dimension}]
\label{thm:lb-unbounded}
Deciding whether $\min\{x^{T}Qx+c^{T}x:\ Ax\le b,\ x\in\Z^{n}\}$ is
unbounded is W[1]-hard parameterized by $n$, with no
$(m+\varphi)^{o(n)}$ algorithm under ETH. Consequently the
$(m+n)^{O(n)}$ bound of Theorem~\ref{thm:unb-main} is optimal up to
a constant factor in the exponent, and no
$f(n)\cdot\mathrm{poly}(m,\varphi)$ unboundedness test exists unless
$\mathrm{FPT}=\mathrm{W[1]}$.
\end{theorem}

\begin{proof}
Take the instance $\min\{4F(z):z\in K\cap\Z^{n+1}\}$ ($b=0$, $c=0$,
feasible since $0\in K$). If some $r\in K$ has $F(r)<0$, then since
$K$ is a rational cone and $\{F<0\}$ is open, a rational and hence
(after clearing denominators) an \emph{integer} $z\in K\cap\Z^{n+1}$
has $F(z)<0$; then $F(tz)=t^{2}F(z)\to-\infty$ over $t\in\Z_{\ge0}$
and the IQP is unbounded. If $F\ge0$ on $K$, the objective is
bounded below by $0$ on all of $K\supseteq K\cap\Z^{n+1}$. So the
IQP is unbounded iff $F$ is not copositive on $K$;
apply Theorem~\ref{thm:lb-copos}. Scale-invariance of the cone is
what makes the integer and continuous questions coincide: there
is no integrality gap to argue.
\end{proof}

\subsection{Scope of the lower bounds}
\begin{table}[t]
\centering\small
\renewcommand{\arraystretch}{1.2}
\begin{tabular}{lll}
\toprule
Target & Dimension & Inertia $(\nu_+,\nu_-,\nu_0)$ \\
\midrule
Continuous or bounded integer threshold & $2k$ & $(0,k,k)$ \\
Noncopositivity or integer unboundedness & $2k+1$ & $(1,k+1,k-1)$ \\
\bottomrule
\end{tabular}
\caption{Transfers from the same multicolored $k$-Sum instance.
Each target has $O(km_0)$ inequalities and coefficient bit length
$O(B+\log k)$ when the source numbers have $B$ bits.}
\label{tab:lower-transfers}
\end{table}
The W[1] statement uses the many-one colored source; the ETH statement
uses the separate bounded-bit source specified above. Both transfers
preserve a linear relation between $k$ and the target dimension.
Thus the lower bound excludes $(m+\varphi)^{o(n)}$ time and matches
the order $O(n)$ in the recession exponent, without fixing its leading
constant. The source coefficients grow with the source input, but their
bit lengths remain polynomial. In particular, these results do not
contradict FPT in $(n,L)$.

The homogenized family has one positive eigenvalue and a growing
negative index. Fixed rank and fixed negative index are different
restrictions, discussed in Section~\ref{sec:discussion}. Taking the
complement changes the noncopositivity statement to co-W[1]-hardness
of copositivity. The construction supplements the known continuous
norm-maximization and IQP hardness~\cite{KKW15,Herrmann26} with an
explicit continuous gap for arbitrary color sets and transfers to
noncopositivity and integer unboundedness with the inertia shown above.

\section{Concave integer minimization}\label{sec:concave}\label{sec:concave-imp}
This section gives a companion result independent of the batching
algorithm. For a concave function on a polytope, some integer optimum
is a vertex of the integer hull. We construct a finite set containing
all such vertices, with a cardinality bound independent of the
constraint right-hand sides. The construction combines proximity
\cite{CookGerardsSchrijverTardos86} with the reflection idea behind
dyadic slack partitions \cite{CookHartmannKannanMcDiarmid92}.
We use an elementary arrangement count, so the explicit exponent is
slightly coarser than the sharper integer-hull counts in that literature.

\subsection{Proximity and slack cells}
Let $P=\{x:Ax\le b,\ Cx=d\}$ be a polytope with integer coefficient
matrices and independent equality rows. For integer feasibility,
rational inequality right-hand sides may first be rounded down.
A noninteger entry of $d$ makes the integer feasible set empty.
Thus we may assume $b,d$ integer. Put
$\Delta=\Delta\bigl(\binom AC\bigr)$ and $n'=n-\rank C$.
All distances and coefficient bounds below refer to the original
ambient coordinates; no lattice-coordinate transformation is made.

\begin{lemma}[Vertex proximity]\label{lem:proximity}
For every vertex $v$ of $P_I=\operatorname{conv}(P\cap\Z^n)$,
there is a vertex $\bar x$ of $P$ with
$\|v-\bar x\|_\infty\le n\Delta$.
\end{lemma}
\begin{proof}
Choose a linear objective uniquely minimized over $P_I$ at $v$,
and choose an optimal vertex $\bar x$ of its relaxation over $P$.
The integer-programming proximity theorem
\cite{CookGerardsSchrijverTardos86} gives an integer optimal
solution within $n\Delta$ of $\bar x$. Uniqueness makes it $v$.
Equality rows can be represented by opposite inequalities without
increasing the maximum absolute subdeterminant.
\end{proof}

\begin{lemma}[Dyadic reflection cells]\label{lem:reflecting}
Let $R=\{x:Tx\le a,\ Cx=d\}$ have integer data, and suppose
$0\le a_i-t_i^Tx<2^M$ throughout $R$, where $M\ge1$.
Partition each integer slack into $\{0\}$ and the ranges
$[2^j,2^{j+1})$, $0\le j<M$.
If a cell contains a vertex $v$ of $R_I$, it contains no other
integer point of $R$.

If $T$ has $r$ rows and $\rank C=n-n'$, a family of at most
\[
 \bigl(1+2r(M+2)\bigr)^{n'}
\]
relatively open polyhedral cells refining this partition covers $R$.
The cells and one integer feasible point from each nonempty cell
can be computed in
$\bigl(1+2r(M+2)\bigr)^{O(n'+1)}n^{O(n)}\mathrm{poly}(\varphi)$
time, where $\varphi$ includes the encoding of $R$ and $M$.
\end{lemma}
\begin{proof}
Suppose distinct integer points $v,y$ lie in the same cell.
For a zero-slack row, both have zero slack. Otherwise their slacks
$s_i(v),s_i(y)$ lie in some $[2^j,2^{j+1})$, so
\[
 s_i(2v-y)=2s_i(v)-s_i(y)>0.
\]
The reflected point $2v-y$ is integral, satisfies the equations,
and satisfies every inequality. Since $v$ is the midpoint of
the distinct feasible integer points $y$ and $2v-y$, it cannot
be a vertex of $R_I$.

Intersect the affine space $Cx=d$ with the hyperplanes
$a_i-t_i^Tx=0,1,2,4,\ldots,2^M$, where the positive values listed
are powers of two. There are at most $H=r(M+2)$ hyperplanes.
The total number of cells of all dimensions in an arrangement of
$H$ hyperplanes in dimension at most $n'$ is at most
$\sum_{j=0}^{n'}2^j\binom Hj\le(1+2H)^{n'}$.
The first inequality follows by inserting hyperplanes successively:
a new hyperplane creates at most twice the number of cells in its
induced arrangement of dimension one less. Restricting cells to $R$
does not increase the count, since $R$ is convex and its boundary
hyperplanes are included. Each resulting cell refines the slack
partition above.

One can enumerate the cells incrementally, retaining the signs
relative to the inserted hyperplanes and using linear programming
to discard empty systems. Strict inequalities are tested by a common
positive slack variable bounded above by $1$. Before integer
feasibility is tested, strict inequalities with integer coefficients
and right-hand sides are converted to non-strict inequalities by
shifting the right-hand side by $1$. A fixed-dimension ILP call
then returns an integer point of each nonempty integer cell. The
number of incremental operations is polynomial in $H$ times the
cell bound; rational linear algebra and the ILP calls have the
stated encoding-length dependence. For $n'=0$, there is at most
one cell, handled by checking its unique affine point.
\end{proof}

\subsection{A right-hand-side-independent candidate set}
\begin{proposition}[Candidate set for integer-hull vertices]\label{prop:vertex-bound}
Suppose $P=\{x:Ax\le b,\ Cx=d\}$ is a polytope with integer data,
$C$ has rank $k$, and $P\cap\Z^n\ne\emptyset$. Set
\[
 \begin{split}
 n'&=n-k,\qquad L_A=\max(1,\|A\|_{\max}),\qquad
 \Delta=\Delta\bigl(\binom AC\bigr),\\
 M&=\left\lceil\log_2(2n^2L_A\Delta+1)\right\rceil,\qquad
 H=(m+2n)(M+2).
 \end{split}
\]
There is a computable set
$\operatorname{vert}(P_I)\subseteq S\subseteq P\cap\Z^n$ with
\[
 |S|\le \max(1,m)^{n'}(1+2H)^{n'}.
\]
It can be constructed in
\[
 \max(1,m)^{n'}(1+2H)^{O(n'+1)}
 n^{O(n)}\mathrm{poly}(\varphi)
\]
time. For $n'=0$ the construction is a single feasibility check.
The cardinality depends on the coefficient matrices but not on $b,d$.
\end{proposition}
\begin{proof}
First handle $n'=0$ by solving $Cx=d$. Otherwise enumerate the
vertices $\bar x$ of $P$ by choosing $n'$ inequality rows independent
modulo $C$, solving the resulting square system, and checking
feasibility. Every vertex occurs among the at most
$\binom m{n'}\le\max(1,m)^{n'}$ choices, even when $P$ has smaller
affine dimension than $Cx=d$.

For each such vertex form the integer-rounded box
\[
 \widehat B_{\bar x}=\{x:
 \lceil\bar x_i-n\Delta\rceil\le x_i\le
 \lfloor\bar x_i+n\Delta\rfloor\quad(i\in[n])\}.
\]
It contains every integer point within distance $n\Delta$ of
$\bar x$ and is contained in that ambient box. Put
$R_0=n^2L_A\Delta$. If an original row has slack greater than
$R_0$ at $\bar x$, it is strictly satisfied throughout the box:
$|a_i^T(x-\bar x)|\le R_0$. Drop these rows. Every retained row
has slack at most $2R_0$ throughout the box wherever it is
satisfied. The $2n$ box rows have slack at most $2n\Delta\le2R_0$.
Thus the remaining rows, the box rows, and $Cx=d$ describe exactly
$R_{\bar x}=P\cap\widehat B_{\bar x}$, with at most $m+2n$
inequalities and all slacks less than $2^M$.

Apply Lemma~\ref{lem:reflecting} to this polytope and collect one
integer point from each cell that has one. If $v$ is a vertex of
$P_I$, proximity places it in some $R_{\bar x}$. A linear functional
exposing $v$ on $P_I$ also exposes it on $(R_{\bar x})_I$, a subset
containing $v$. Its cell therefore contains no other feasible
integer point, so the ILP call for that cell returns $v$.
All collected points remain feasible in the original $P$.
Multiplying the per-box count and time by the number of boxes
proves the bounds. Box endpoints and cell right-hand sides have
polynomial encoding length in the input; large $b,d$ enter only
through this arithmetic.
\end{proof}

\begin{proposition}[Concave integer minimization]\label{prop:concave}
Let $P=\{x:Ax\le b\}$ be a rational polytope with integer $A$,
and let $f$ be concave with exact rational evaluation in time
polynomial in the encoding lengths of its description and argument.
Minimizing $f$ on $P\cap\Z^n$ takes
\[
 ((m+n)n\varphi_A)^{O(n)}\mathrm{poly}(\varphi)
\]
time. With additional independent integer equations $Cx=d$, the
same conclusion holds with $\varphi_{A,C}$ replacing $\varphi_A$.
\end{proposition}
\begin{proof}
Test integer feasibility first. If feasible, construct $S$ from
Proposition~\ref{prop:vertex-bound} and return a point minimizing
$f$ on $S$. A concave function on the polytope $P_I$ has a vertex
minimizer, and $S$ contains every such vertex, so the answer is exact.
When $C$ is absent,
$\Delta(A)\le(nL_A)^n$ and $\log L_A\le\varphi_A$, giving
$M=O(n(\varphi_A+\log n))$. Substitution in the construction bound
yields the displayed runtime. With equations, use the corresponding
entry and determinant bounds for the combined coefficient matrix.
\end{proof}
The additional $(m+n)$ accommodates the ambient box rows. For a
nonempty bounded polyhedron with no equality rows and $n\ge1$,
$m\ge n$, so this is also $(mn\varphi_A)^{O(n)}\mathrm{poly}(\varphi)$.
The explicit count in Proposition~\ref{prop:vertex-bound}
retains the reduced affine dimension $n'$ without changing coordinates.

\section{Applications}\label{sec:applications}
Theorem~\ref{thm:main} gives an explicit cost for any IQP subroutine
once its dimension, coefficient bounds, and outer enumeration count
are specified. We first record structured formulations, then give
these quantities for graph applications. All graph bounds below
assume $k\ge1$; the zero-parameter cases are direct.

\subsection{Structured optimization formulations}
Corollary~\ref{cor:tu} applies to quadratic objectives over integral
network-flow or transportation constraints. For example, directed
incidence equations together with integral arc bounds form a totally
unimodular system. Arbitrary quadratic arc costs can therefore be
handled with the runtime in that corollary; separable convex costs
have more specialized algorithms \cite{Minoux86,AhujaHochbaumOrlin03}.
Production-planning models \cite{PochetWolsey06} must be assessed by
their actual formulation: adding setup or coupling constraints can
change total unimodularity. Regularized integer least squares is
another IQP formulation \cite{AgrellErikssonVardyZeger02}.

\subsection{Graph algorithms and their parameters}
For a vertex cover or clique modulator $X$ of size $k$, partition
$V(G)\setminus X$ by neighborhoods in $X$. Let $T\le2^k$ be the
number of nonempty types. In arrangement formulations, count variables
specify how many vertices of each type lie in each of the $k+1$ gaps
between ordered modulator vertices. Instances with $T=0$ are handled directly. Auxiliary position variables are
included in the IQP dimension whenever they are used.

\begin{table}[H]
\centering\small
\renewcommand{\arraystretch}{1.25}
\begin{tabular}{P{0.31\textwidth}P{0.15\textwidth}P{0.19\textwidth}P{0.23\textwidth}}
\toprule
Problem and parameter & IQP variables & Coefficients & Outer enumeration \\
\midrule
Optimal Linear Arrangement; vertex cover $k$ &
$O(kT)$ & $L_A=1$, $\Delta=1$, $L_Q=O(k^2 2^k)$ & $k!$ \\
Exact Crossing Number; vertex cover $k$ &
$T'\le2^k(k-1)!$ & $L_A=1$, $\Delta=1$, $L_Q=O(k^2)$ &
$2^{k^{O(k)}}$ cluster drawings \\
Minimum Sum Vertex Cover; clique modulator $k$ &
$O(kT)$ & $L_A=1$, $L_Q=O(k)$ & $k!$ \\
\bottomrule
\end{tabular}
\caption{Formulation parameters for the cited reductions. Objective
polynomials are multiplied by a fixed constant where needed to obtain
symmetric integer matrices. The graph size enters the encoding length
polynomially. The outer enumeration must be included in the total time.}
\label{tab:applications}
\end{table}

\paragraph{Optimal Linear Arrangement.}
This problem minimizes $\sum_{uv\in E(G)}|\sigma(u)-\sigma(v)|$
over vertex orderings $\sigma$. Lokshtanov's reduction
\cite[Section~3]{Lokshtanov17} uses $k!$ cover orders,
$O(kT)$ gap-count variables, and a quadratic objective with
half-integral coefficients bounded by $O(k^2 2^k)$.
Its constraints are nonnegativity and one sum equation per type.
These disjoint sum rows together with unit rows are TU, and
$m=O(kT)$. Scaling the objective and applying
Corollary~\ref{cor:tu} yields
\[
 k!\,(kT\cdot k^2 2^k)^{O(kT)}\,|G|^{O(1)}.
\]
Here and below $|G|$ denotes the graph encoding length.

\paragraph{Exact Crossing Number.}
For simple graphs, this problem minimizes the number of edge crossings
in a drawing. Hlin\v{e}n\'{y}--Sankaran
\cite[Section~4]{HlinenySankaran19} enumerate abstract topological
cluster drawings and assign integer multiplicities to their clusters.
There are at most $T'=2^k(k-1)!$ clusters. The assignment constraints
are again disjoint sum equations and nonnegativity, hence TU.
Each cluster has at most $k$ incident edges; a good drawing has at
most one crossing per edge pair, giving $L_Q=O(k^2)$.
Consequently each IQP costs
\[
 (T'k)^{O(T')}\,|G|^{O(1)}.
\]
Multiplying by the $2^{k^{O(k)}}$ possible cluster drawings from the
cited enumeration gives a total of $2^{k^{O(k)}}|G|^{O(1)}$.
The explicit IQP cost is only one part of this total.

\paragraph{Minimum Sum Vertex Cover.}
The objective is $\sum_{uv\in E(G)}\min(\sigma(u),\sigma(v))$.
Aute--Panolan's IQP reduction
\cite[Section~4]{AutePanolan24} uses a clique modulator: deleting
its $k$ vertices leaves a clique. Keeping their auxiliary count,
position, and degree variables gives $O(kT)$ variables, coefficient
rows in $\{0,\pm1\}$, and quadratic coefficients of size $O(k)$.
After the $k!$ modulator orders are enumerated,
Theorem~\ref{thm:main} gives
\[
 k!\,(kT\cdot k)^{O(k^2T^2)}\,|G|^{O(1)}.
\]
Their vertex-cover algorithm is a separate combinatorial enumeration
with an explicit runtime; it is not an application of this IQP bound.

\paragraph{Densest $\kappa$-Subgraph as a formulation example.}
With a vertex cover $X$ of size $k$, use binary variables $z_i$ for
cover vertices and counts $w_t$ for independent-set types. Bounds
$0\le w_t\le|I_t|$ and $\sum_i z_i+\sum_t w_t=\kappa$, together
with the binary bounds, form a TU system. The number of selected
edges is
\[
 \sum_{\{i,j\}\in E(G[X])}z_i z_j
 +\sum_{i,t:\,i\in N(I_t)}z_i w_t.
\]
Negating and scaling gives an IQP with $k+T$ variables and bounded
quadratic coefficients, solvable in
$(k+T)^{O(k+T)}|G|^{O(1)}$ by Corollary~\ref{cor:tu}.
For this parameter a direct $2^k|G|^{O(1)}$ algorithm is faster:
fix the selected subset of $X$, then choose the required number of
independent vertices with largest numbers of neighbors in that
subset. Thus the formulation illustrates the theorem's applicability,
without providing an improved algorithm for this problem.

IQP also appears inside algorithms with structural restrictions
\cite{EibenGanianKnopOrdyniak19}. The same substitution applies once
the number and parameters of those subroutine calls are accounted for.

\section{Discussion and open problems}\label{sec:discussion}
Curvature batching confines numerical branching to a phase in which
the coefficient growth is controlled, followed by a single batch and
a linear leaf. The recession test uses a different mechanism and is
polynomial in the input length for every fixed dimension. These
results lead to four questions.

\paragraph{Finite-value optimization without coefficient-magnitude parameters.}
Can IQP with a finite optimum be solved in polynomial time for every
fixed $n$, with dependence on $A,Q$ only through their encoding lengths?
The unboundedness test removes infeasibility and recession from this
question, but the batching algorithm still enumerates numerical slack
and gradient values. In the mixed case, a related issue is extending
the stated polytope algorithm to arbitrary polyhedra with finite
optimum while preserving the same coefficient parameters.

\paragraph{Fixed rank and a fixed negative index.}
What is the parameterized complexity of bounded IQP at inertia
$(1,1,n-2)$, and of noncopositivity with a fixed negative index at
least two? Our homogenized reductions have inertia $(1,k+1,k-1)$:
they fix the positive index while the negative index and rank grow.
A low-dimensional image of the integer feasible set may be useful
\cite{EisenbrandRothvoss25}, but an algorithm would need bounds on
its construction cost and representation size, including dependence
on all slice right-hand sides.

\paragraph{Dependence on continuous dimension.}
Can the mixed batching exponent $O(n^2(q+1)^2)$ be reduced to
$O(n^2)$ uniformly in $q$? The current cost comes from the coefficients
of the projected rows and their cumulative determinant growth after
the continuous phase. The common-denominator construction already
prevents further denominator growth, so sharper projection bounds
are a natural next target.

\paragraph{Constraint-count dependence.}
Can one reduce the number of slack branches or the size of the
recession refinement? ETH rules out an exponent $o(n)$ for the
recession test, but the leading constant in the exponent remains
undetermined. The relevant count includes the whole decomposition
and all slope tests. For TU systems, sharper bounds on the geometry
of the feasible region may also improve the constraint phase or the
concave candidate-set construction.

These questions distinguish fixed-dimensional polynomial time from
FPT in the dimension. In particular, a continuous copositivity test
alone cannot decide integer unboundedness: integer feasibility and
flat slopes remain. Given an ILP $Az\le b$, minimizing $-t$ subject
to that system and an integer $t\ge0$ is unbounded exactly when the
ILP is feasible, although its quadratic matrix is zero.

\appendix
\section{Sequential branching and inertia}\label{sec:seq}
This appendix compares batching with a sequential variant of
Lokshtanov's framework. The variant uses the adjugate basis of
Lemma~\ref{lem:adjugate}, retains only nonnegative-curvature gradient
branches, and solves flat leaves by ILP. Its bound exposes the cost
of repeated determinant squaring; it is not needed for the batching
algorithm.

\begin{theorem}[Sequential bound]\label{thm:seq}
IQP can be solved in time $(nL)^{O(n^2\cdot 2^{\nu_+})}\cdot\mathrm{poly}(\varphi)$,
where $\nu_+=\nu_+(Q)$.
\end{theorem}

Recall the two improvements: (i)~an adjugate basis
with $\|y_i\|_\infty\le\Delta(C)$ instead of $\Delta(C)^2$, reducing the
growth per gradient step from cubing to squaring; and (ii)~a positive eigenvalue bound
showing that at most $\nu_+$ gradient steps can occur on any root-to-leaf
path.

\subsection{Positive eigenvalues and gradient steps}
A gradient branch has $y^TQy\ge0$ by Lemma~\ref{lem:deepcurv}.
We first characterize the existence of such a direction in the whole
kernel; this does not assert that every chosen basis contains one.

\begin{lemma}[Gradient branching characterization]\label{lem:gradiff}
A direction with nonnegative curvature and a gradient row outside
the current row space exists in $V$ if and only if\/ $\nu_+(Q|_V)\ge 1$.  That is, there exists
$y\in\ker(C)$ with $y^T Q\notin\rowsp(C)$ and $y^T Qy\ge 0$ if and
only if\/ $\nu_+(Q|_V)\ge 1$.
\end{lemma}

\begin{proof}
\medskip\noindent\emph{($\Rightarrow$)}
Suppose $y\in V$ satisfies $y^T Q\notin\rowsp(C)$ and $y^T Qy\ge 0$.
If $y^T Qy>0$, then $\nu_+(Q|_V)\ge 1$ immediately.
If $y^T Qy=0$, then since $y^T Q\notin\rowsp(C)$, there exists
$w\in V$ with $y^T Qw\neq 0$.  For every $t\in\R$ the vector $y+tw$
lies in~$V$, so
\[
  p(t)\;:=\;(y+tw)^T Q(y+tw)
       \;=\;2t\,(y^T Qw)+t^2\,(w^T Qw).
\]
Since $p(0)=0$ and $p'(0)=2\,y^T Qw\neq 0$, choosing $t$ small with
the same sign as $y^T Qw$ gives $p(t)>0$, so $\nu_+(Q|_V)\ge 1$.

\medskip\noindent\emph{($\Leftarrow$)}
Suppose $\nu_+(Q|_V)\ge 1$, so there exists $v\in V$ with $v^T Qv>0$.
If $v^T Q\in\rowsp(C)$, then $v^T Q=\lambda^T C$ for some~$\lambda$, and
for every $w\in V$ we would have $v^T Qw=\lambda^T Cw=0$.  In particular
$v^T Qv=0$, contradicting $v^T Qv>0$.  Therefore $v$ satisfies both
$v^T Q\notin\rowsp(C)$ and $v^T Qv\ge 0$, providing a valid gradient
direction.
\end{proof}

\begin{lemma}\label{lem:inertia}
Each gradient branching strictly reduces $\nu_+(Q|_V)$. That is, for $y \in \ker(C)$ with $y^T Q\notin\rowsp(C)$
and $y^T Qy\ge 0$, adding the row $2y^TQ$ to $C$ strictly reduces the number of positive eigenvalues of $Q$ restricted to the new kernel.
\end{lemma}

\begin{proof}
When branching on a gradient direction, we have $y\in V$ with $y^T Q\notin\rowsp(C)$
and $y^T Qy\ge 0$, and we add the row $2y^TQ$ to $C$. This means we pass to the smaller subspace
$W=V\cap\ker(y^T Q)=\{w\in V : y^T Qw=0\}$.  We show
$\nu_+(Q|_W)\le\nu_+(Q|_V)-1$ when $y^T Qy>0$ and when $y^T Qy=0$.

\medskip\noindent\emph{Case $y^T Qy>0$.}  Since $y^T Q y > 0$, the vector $y$ does not lie in $W$ (because $y^T Q y \neq 0$).  Since $W$ has dimension $r-1$ and $y \notin W$, we can form a basis for $V$ by taking $y$ alongside any basis $\{w_1, \ldots, w_{r-1}\}$ for $W$:
\[
  B_V = [y \mid w_1 \mid \cdots \mid w_{r-1}].
\]

Compute $B_V^T Q B_V$:
\[
  B_V^T Q B_V =
  \begin{pmatrix} y^T \\ B_W^T \end{pmatrix}
  Q
  \begin{pmatrix} y & B_W \end{pmatrix}
  =
  \begin{pmatrix} y^T Q y & y^T Q B_W \\ B_W^T Q y & B_W^T Q B_W \end{pmatrix},
\]
where $B_W = [w_1 \mid \cdots \mid w_{r-1}]$.  Every column $w_i$ of $B_W$ lies in $W$, so by definition $y^T Q w_i = 0$.  Therefore $y^T Q B_W = 0$ (the entire row of cross terms vanishes), and
\[
  B_V^T Q B_V =
  \begin{pmatrix} y^T Qy & 0 \\ 0 & B_W^T Q B_W \end{pmatrix}.
\]

Since $y^T Qy > 0$, we have $\nu_+(B_W^T Q B_W) = \nu_+(B_V^T Q B_V) - 1$.

\medskip\noindent\emph{Case $y^T Qy=0$.}  
As $y^T Q \notin \rowsp(C)$, $y^TQ$ is not identically zero on $V$. Then, there exists $z \in V$ with $y^T Q z \neq 0$.  After rescaling, assume $y^T Q z = 1$. 

Define
\[
  z' = z - \tfrac{1}{2}(z^T Q z)\, y.
\]
Then, we have $y^T Q z' =  1$ 
and $z'^T Q z' = 0$. Next, we construct a basis $B_V$ for $V$. Since $y$ and $z'$ are linearly independent, we include them in the basis. We complete the basis by picking $u_1, \dots, u_{r-2}$ from 
$W \cap \ker(z'^T Q)$, which has dimension $r - 2$ by rank-nullity 
(since $z'^T Q$ is nonzero on $W$ due to $z'^T Q y = 1$).

\[
  B_V = [y \mid z' \mid u_1 \mid \cdots \mid u_{r-2}].
\]

By construction, $y^T Q u_i = 0$ and $z'^T Q u_i = 0$ for every $i$, so $B_V^T Q B_V$ is block diagonal:
\[
  B_V^T Q B_V =
  \begin{pmatrix}
    0 & 1 & 0 \\
    1 & 0 & 0 \\
    0 & 0 & E
  \end{pmatrix},
\]
where $E_{ij} = u_i^T Q u_j$. The top-left block $\bigl(\begin{smallmatrix} 0 & 1 \\ 1 & 0 \end{smallmatrix}\bigr)$ has eigenvalues $\pm 1$, contributing exactly one positive eigenvalue.  Therefore
\[
  \nu_+(B_V^T Q B_V) = 1 + \nu_+(E).
\]

Finally, observe that every $u_i$ satisfies $y^T Q u_i = 0$, so $u_i \in W$.  The vector $y$ itself satisfies $y^T Q y = 0$, so $y \in W$.  Thus $\{y, u_1, \ldots, u_{r-2}\}$ are $r - 1$ linearly independent vectors in $W$ (they are a subset of the basis $B_V$).  Since $\dim(W) = r - 1$, they form a basis $B_W = [y \mid u_1 \mid \cdots \mid u_{r-2}]$. Then, we have

\[
  B_W^T Q B_W =
  \begin{pmatrix}
    0 & 0 \\
    0 & E
  \end{pmatrix}.
\]

The zero row/column contributes nothing to $\nu_+$, so
\[
  \nu_+(B_W^T Q B_W) = \nu_+(E) = \nu_+(B_V^T Q B_V) - 1. \qedhere
\]
\end{proof}

By Lemmas~\ref{lem:deepcurv},~\ref{lem:gradiff}, and~\ref{lem:inertia}, we can conclude that the total number of gradient steps on any root-to-leaf path is at most $\nu_+(Q)$: gradient branching requires $y^T Qy\ge 0$ (Lemma~\ref{lem:deepcurv}), which is only possible when $\nu_+(Q|_V)\ge 1$ (Lemma~\ref{lem:gradiff}), and each such branch consumes one positive eigenvalue (Lemma~\ref{lem:inertia}), so after $\nu_+$ steps no further gradient branch is possible.

\subsection{Running time analysis}
\label{sec:lokbound}

Our sequential algorithm is Lokshtanov's branching algorithm equipped with two
improvements: the adjugate basis from Section~\ref{sec:adjugate} and the
positive eigenvalue bound from Lemma~\ref{lem:inertia}.

\begin{algorithm}[H]
\caption{Sequential IQP Solver}\label{alg:seq}

\small
\textbf{Input:} $Q,c,A,b$ and independent equalities $Cx=d$.
\textbf{Precondition:} The integer objective is bounded below on the
feasible set (which may be empty). First apply
Theorem~\ref{thm:unb-main} to classify the original instance.
\textbf{Output:} The optimum on this node, or $+\infty$ if infeasible.
\begin{enumerate}[leftmargin=*,itemsep=3pt]
\item Discard an inconsistent affine system. Put $V=\ker C$ and
$r=\dim V$. If $r=0$, evaluate the unique solution only after
checking $Ax\le b$ and integrality; return.
\item Compute an adjugate basis $y_1,\ldots,y_r$ with
$\|y_i\|_\infty\le\Delta(C)$.
\item If $y_i^TQ\in\rowsp C$ for all $i$, solve the residual ILP
with the objective of Corollary~\ref{cor:linear}, and return its
quadratic value, or $+\infty$ if infeasible.
\item Initialize $\mathit{best}=+\infty$. For each
$a_j\notin\rowsp C$ and integer
$b_j'\in[b_j-nL_A\Delta(C),b_j]$, recurse after appending
$a_j^Tx=b_j'$; update $\mathit{best}$ with the returned value.
\item For each basis vector $y_i$ with
$y_i^TQ\notin\rowsp C$ and $y_i^TQy_i\ge0$, and each integer
$z$ with $|z|\le y_i^TQy_i$, recurse after appending
$2y_i^TQx=z-c^Ty_i$; update $\mathit{best}$.
\item Return $\mathit{best}$.
\end{enumerate}

\end{algorithm}

\begin{proof}[Proof of Theorem~\ref{thm:seq}]
First apply Theorem~\ref{thm:unb-main} and return immediately if the
instance is infeasible or unbounded below. Its
$(m+n)^{O(n)}\mathrm{poly}(\varphi)$ cost, after duplicate coefficient
rows have been removed while retaining the smallest right-hand side,
is absorbed in the claimed bound. This preprocessing itself is
polynomial in the input length. For a finite optimum, invoke
Algorithm~\ref{alg:seq} with no initial equalities.
We analyze its worst-case branching tree.
On any root-to-leaf path, each step adds one row to~$C$ and reduces
$\dim(\ker(C))$ by one, so there are at most $n$~steps total.  By
Lemmas~\ref{lem:gradiff} and~\ref{lem:inertia}, at most
$g \le \nu_+$ of these are gradient steps; if there are $s$
constraint steps, then $s+g\le n$, and in particular $s\le n$. 

To bound the determinants without assuming an order of steps, choose
$B\ge2$ with $nL_A,2n^2L_Q\le B=(2+nL)^{O(1)}$ and write
$e=\log_B\Delta$. A constraint step sends $e$ to at most $e+1$;
a gradient step sends it to at most $2e+1$. Over $s$ constraint
steps and $g$ gradient steps, in any order,
\[
 e\le 2^g(e_0+s+1)-1.
\]
This follows by assigning each additive increment the multiplier
from the gradient steps still to come. Thus only $s\le n$ and
$g\le\nu_+$ are needed.

The initial matrix consists of rows of $A$, so
$\Delta(C_0)\le(nL)^{O(n)}$ and $e_0=O(n)$. Substituting
$s\le n$ and $g\le\nu_+$ into the recurrence gives
\begin{equation}\label{eq:global-bound}
 \Delta_{\max}\le(nL)^{O(n2^{\nu_+})}.
\end{equation}

Next, we bound the branching factor. At a constraint step the algorithm branches over which row
$a_j \notin \rowsp(C)$ to add (at most $m$ choices) and which value
$b_j' \in \{b_j - nL_A\Delta, \ldots, b_j\}$ to assign (at most
$nL_A\Delta + 1$ choices), giving a branching factor of at most
$m \cdot (nL_A\Delta + 1) \le (nL)^{O(n)} \cdot \Delta$, since
$m \le (2L_A+1)^n = (nL)^{O(n)}$.  For each eligible basis vector, a gradient step creates at most
$2\,y_i^T Qy_i + 1 \le 2n^2 L_Q \Delta^2 + 1 =
(nL)^{O(1)} \cdot \Delta^2$.  There are at most $n$ eligible basis vectors. Therefore the total number of children
at any node is at most $(nL)^{O(n)} \cdot \Delta^2$.
Substituting~\eqref{eq:global-bound}, the branching factor at every
node is at most $(nL)^{O(n \cdot 2^{\nu_+})}$.

Since the tree has depth at most~$n$ and branching factor is at most
$(nL)^{O(n\cdot 2^{\nu_+})}$ at each node, the total number of
leaves is at most
$\bigl[(nL)^{O(n\cdot 2^{\nu_+})}\bigr]^n = (nL)^{O(n^2\cdot 2^{\nu_+})}$.
Each leaf is a checked point or an ILP, solved by fixed-dimension integer linear programming~\cite{Kannan87} in
$n^{O(n)} \cdot \mathrm{poly}(\varphi)$ time, which is absorbed into
the tree-size exponent.  Node coefficients have bit length polynomial in $\varphi$ and
$n2^{\nu_+}\log(nL)$; their additional arithmetic cost is absorbed
in the same parametric factor. The overall running time is
$(nL)^{O(n^2\cdot 2^{\nu_+})} \cdot \mathrm{poly}(\varphi)$.
\end{proof}

\medskip
\noindent\textbf{Comparison with the original algorithm.}
Lokshtanov constructs integer nullspace vectors with norm at most
$\Delta(C)^2$; lattice generation is not required by his argument
\cite[Section~2]{Lokshtanov17}. Substituting that bound into
Lemma~\ref{lem:growth} gives cubing rather than squaring of
$\Delta$. For a variant with ILP leaves, the same tree analysis
then gives $(nL)^{O(n^2 3^n)}\mathrm{poly}(\varphi)$.
This is not a bound for the literal neighborhood-enumeration step
in the original presentation. There the flat alternative is part
of an existential three-case lemma, and the algorithm also searches
integer points within $\ell_1$ distance $n\Delta^2$ of a particular
affine solution. Bounding this search by its containing coordinate
box gives at most $(2n\Delta^2+2)^n$ candidates. This refined count
also fits $(nL)^{O(n^2 3^n)}\mathrm{poly}(\varphi)$ once
$\Delta\le(nL)^{O(n3^n)}$ is substituted. The larger enumeration
bound used in the original FPT analysis is unnecessary for an
explicit comparison. In either implementation, the determinant
recurrence is the source of the double-exponential dependence.

\bibliographystyle{plainnat}
\bibliography{references}

\end{document}